\RequirePackage[l2tabu, orthodox]{nag}

\documentclass[12pt]{amsart}
\usepackage{fullpage,url,amssymb,enumerate,colonequals}
\usepackage[all]{xy} 
\usepackage{comment}
\usepackage{graphicx}

\usepackage[OT2,T1]{fontenc}

\usepackage{color}

\def\Adm{\mathcal{A}dm}

\def\bC{\mathbb{C}}
\def\cC{\mathcal{C}}

\def\cH{\mathcal{H}}

\def\cM{\mathcal{M}}

\def\cO{\mathcal{O}}
\def\bP{\mathbb{P}}

\def\bZ{\mathbb{Z}}

\def\barM{\overline{\cM}}

\def\wt{\widetilde}

\DeclareMathOperator{\id}{id}

\DeclareMathOperator{\pr}{pr}

\DeclareMathOperator{\Qmod}{Qmod}

\DeclareMathOperator{\Spec}{Spec}

\newtheorem{thm}{Theorem}[section]
\newtheorem{lem}[thm]{Lemma}
\newtheorem{cor}[thm]{Corollary}
\newtheorem{prop}[thm]{Proposition}

\newtheorem{conjecture}{Conjecture}

\theoremstyle{definition}
\newtheorem{rem}[thm]{Remark}

\newtheorem{defn}[thm]{Definition}

\makeatletter
\g@addto@macro\bfseries{\boldmath} 
\makeatother

\usepackage{microtype}

\usepackage[
	backref,
	pdfauthor={Carl Lian}, 
	pdftitle={$d$-elliptic loci in genus 2 and 3},
]{hyperref}

\begin{document}

\title{$d$-elliptic loci in genus 2 and 3}

\author{Carl Lian}
\address{Institut f\"{u}r Mathematik, Humboldt-Universit\"{a}t zu Berlin, Berlin, 10965, Germany}
\email{liancarl@hu-berlin.de}
\urladdr{\url{http://sites.google.com/view/carllian/}}

\date{\today}

\begin{abstract}
We consider the loci of curves of genus 2 and 3 admitting a $d$-to-1 map to a genus 1 curve. After compactifying these loci via admissible covers, we obtain formulas for their Chow classes, recovering results of Faber-Pagani and van Zelm when $d=2$. The answers exhibit quasimodularity properties similar to those in the Gromov-Witten theory of a fixed genus 1 curve; we conjecture that the quasimodularity persists in higher genus, and indicate a number of possible variants.
\end{abstract}

\maketitle


\section{Introduction}\label{intro}

\textbf{In the published version of this paper, there is an error in the proof of Proposition 5.1, which leads to mistakes in the formulas of Theorem 5.2, Proposition 6.6, and Theorem 1.4. These formulas have been corrected in this version. See the erratum \cite{erratum} for more details, linked on the author's webpage.}

Let $\cH_{g/1,d}$ denote the moduli space of degree $d$ covers $f:C\to E$, where $C$ is a smooth curve of genus $g$, $E$ is a smooth curve of genus 1, and $f$ is simply branched at marked points $x_1,\ldots,x_{2g-2}\in C$ mapping to distinct points $y_1,\ldots,y_{2g-2}\in E$. We then have a diagram
\begin{equation*}
\xymatrix{
\cH_{g/1,d} \ar[r]^{\pi_{g/1,d}} \ar[d]^{\psi_{g/1,d}} & \cM_{g} \\
\cM_{1,2g-2}
}
\end{equation*}
where the map $\pi_{g/1,d}$ remembers the source curve, and $\psi_{g/1,d}$ remembers the target curve with the branch points. The degree of $\psi_{g/1,d}$ is given by a \textit{Hurwitz number}, which counts monodromy actions of $\pi_1(E-\{y_1,\ldots,y_{2g-2}\},y)$ on the $d$-element set $f^{-1}(y)$. We have the following groundbreaking result of Dijkgraaf:

\begin{thm}[\cite{dijkgraaf}]\label{dijkgraaf}
For $g\ge2$, the generating series
\begin{equation*}
\sum_{d\ge1}\deg(\psi_{g/1,d})q^d
\end{equation*}
is a quasimodular form of weight $6g-6$.
\end{thm}

Eskin-Okounkov \cite{eskin} generalized this result to covers of elliptic curves with arbitrary branching. In a slightly different direction, Okounkov-Pandharipande \cite{op} gave a generalization to the Gromov-Witten theory of an elliptic curve with arbitrary insertions, which was later upgraded to a cycle-theoretic statement by Oberdieck-Pixton \cite{obpix18}.

In this paper, we consider instead the enumerative properties of the map $\pi_{g/1,d}$. Here, the geometry is more subtle, as we no longer have a combinatorial model as in Hurwitz theory. For enumerative applications, one needs to compactify the moduli spaces involved; we pass to the Harris-Mumford stack $\Adm_{g/1,d}$ of admissible covers, see \cite{hm} or \S\ref{admissible_covers_prelim}. We have the diagram:
\begin{equation*}
\xymatrix{
\Adm_{g/1,d} \ar[r]^(0.58){\pi_{g/1,d}} \ar[d]^{\psi_{g/1,d}}& \barM_{g} \\
\barM_{1,2g-2}
}
\end{equation*}

We prove:

\begin{thm}\label{modularity_statement}
For $g=2,3$, we have:
\begin{equation*}
\sum_{d\ge1}[(\pi_{g/1,d})_{*}(1)]q^d\in A^{g-1}(\barM_g)\otimes\Qmod,
\end{equation*}
where $\Qmod$ is the ring of quasimodular forms.
\end{thm}

More precisely, we have the following formulas, where $\sigma_k(d)$ is the sum of the $k$-th powers of the divisors of $d$:

\begin{thm}\label{main_thm_g2}
The class of $(\pi_{2/1,d})_{*}(1)$ in $A^{1}(\barM_2)$ is:
\begin{equation*}
\left(2\sigma_3(d)-2d\sigma_1(d)\right)\delta_0+\left(4\sigma_3(d)-4\sigma_1(d)\right)\delta_1.
\end{equation*}
\end{thm}

\begin{thm}\label{main_thm_g3}
The class of $(\pi_{3/1,d})_{*}(1)$ in $A^{2}(\barM_3)$ is:
\begin{align*}
&\left((-4464d^2+2880d-360)\sigma_1(d)+(4092d-2400)\sigma_3(d)+252\sigma_5(d)\right)\lambda^2\\
+&\left((864d^2-288d+36)\sigma_1(d)+(-852d+240)\sigma_3(d)\right)\lambda\delta_0\\
+&\left((1440d^2+864d-24)\sigma_1(d)+(-1320d-960)\sigma_3(d)\right)\lambda\delta_1\\
+&\left((-36d^2)\sigma_1(d)+(36d)\sigma_3(d)\right)\delta_0^2\\
+&\left((-144d^2-120d+12)\sigma_1(d)+(132d+120)\sigma_3(d)\right)\delta_0\delta_1\\
+&\left((-144d^2-288d+60)\sigma_1(d)+(132d+240)\sigma_3(d)\right)\delta_1^2\\
+&\left((144d^2-288d+36)\sigma_1(d)+(-132d+240)\sigma_3(d)\right)\kappa_2.
\end{align*}
\end{thm}

As a check, both formulas become zero after substituting $d=1$. When $d=2$, we recover the main results of Faber-Pagani \cite{faber}. Indeed, the morphism $\pi_{2/1,2}$ is generically 4-to-1, where one factor of 2 comes from the complement $C\to E_2$ to any bielliptic map $C\to E_1$ (see \cite[\S 2]{kuhn}), and another comes from the labelling of the two ramification points. Thus, our answer differs from \cite[Proposition 2]{faber} by a factor of 4. In genus 3, the morphism $\pi_{3/1,2}$ has degree $4!$, coming from the ways of labelling the ramification points of a bielliptic cover, and our answer differs from the correction to \cite[Theorem 1]{faber} given by van Zelm \cite[(3.5)]{vanzelm_thesis} by a factor of 24.

We are then led to conjecture:

\begin{conjecture}\label{main_conj}
The statement of Theorem \ref{modularity_statement} holds for all $g\ge2$.
\end{conjecture}

In fact, our method, which we now outline, suggests a number of possible refinements of Conjecture \ref{main_conj}. For $g=2,3$, it suffices to intersect $\pi_{g/1,d}$ with test classes of complementary dimension, all of which lie in the boundary of $\barM_g$. The test classes may then be moved to general cycles in a boundary divisor of $\barM_g$. The intersection of general boundary cycles with the admissible locus can be expressed in terms of contributions from admissible covers of a small number of topological types, and we compute these contributions in terms of branched cover loci in lower genus. This leads naturally to a number of auxiliary situations in which quasimodularity phenomena also occur, for example:
\begin{itemize}
\item Loci of $d$-elliptic curves with marked points having equal image under the $d$-elliptic map, see \S\ref{m12_d-ell}
\item Loci of curves covering a \textit{fixed} elliptic curve, see \S\ref{m2_d-ell_fixed_target} and also \cite{op}
\item Correspondence maps $(\pi_{g/1,d})_{*}\circ\psi_{g/1,d}^{*}$, see \S\ref{m2_correspondence}
\item Loci of $d$-elliptic curves with marked ramification points, see \S\ref{m21_d-ell} and Appendix \ref{m22_appendix}
\end{itemize}

In addition, $d$-elliptic cycles with non-simple branching are studied using similar methods to ours in \cite{chen}; analogous quasi-modularity phenomena arise here as well.

Let us mention one consequence of Conjecture \ref{main_conj}. Building on work of Graber-Pandharipande \cite{gp}, van Zelm \cite{vanzelm} has shown that the class $(\pi_{g/1,2})_{*}(1)\in H^{2g-2}(\barM_g)$ is non-tautological for $g\ge12$, and that $(\pi_{g/1,2})_{*}(1)\in H^{2g-2}(\cM_g)$ is non-tautological for $g=12$. Assuming Conjecture \ref{main_conj}, we have that for these values of $g$, the generating functions 
\begin{equation*}
\sum_{d\ge1}(\pi_{g/1,d})_{*}(1)q^d
\end{equation*}
in the quotients of $A^{g-1}(\barM_g)$ and $A^{g-1}(\cM_g)$ by their tautological subgroups are nonzero quasimodular forms (note that the tautological part of $A^{g-1}(\cM_g)$ is zero by Looijenga's result \cite{looijenga}). In particular, we would get infinitely many non-tautological classes from $d$-elliptic loci on $\barM_g$ for $g\ge12$, and on $\cM_{12}$.

While our methods are too limited to approach Conjecture \ref{main_conj} in full generality, we remark on one possible approach relating to previous work. Let $g:X\to B$ be a non-isotrivial elliptic surface. Then, the $d$-elliptic loci may be expressed in terms of the space of stable maps with image equal to $d$ times the class of a fiber of $g$. One may then hope to relate quasi-modularity of $d$-elliptic loci to to the quasi-modularity results for Gromov-Witten invariants of elliptic fibrations of \cite{obpix19}. However, the moduli spaces of stable maps in question typically fail to have the expected dimension, so such a relationship would necessarily be subtle.

The structure of this paper is as follows. We collect preliminaries in \S\ref{prelim}, recording the needed facts about intersection theory on $\barM_{g,n}$ and recalling the definitions of admissible covers and quasimodular forms. In \S\ref{prelude}, we carry out some enumerative calculations for branched covers that we will need when considering $d$-elliptic loci. In \S\ref{m2_d-ell}, we prove Theorem \ref{main_thm_g2} on the $d$-elliptic loci in genus 2; it is here where we explain our method in the most detail. In \S\ref{g2_var}, we establish variants of Theorem \ref{main_thm_g2} suggesting possible variants of Conjecture \ref{main_conj}. Finally, we put together all of the previous results to prove Theorem \ref{main_thm_g3} on the $d$-elliptic loci in genus 3 in \S\ref{m3_d-ell}.

We remark in Appendix \ref{m22_appendix} that we have quasimodularity for $d$-elliptic loci on $\barM_{2,2}$, where the ramification points of a $d$-elliptic cover are marked. However, we explain a new feature: not all contributions to the classes of the $d$-elliptic loci from admissible covers of individual topological types are themselves quasimodular.

\subsection{Acknowledgments}
I am grateful to my advisor Johan de Jong, who suggested the initial research directions and offered invaluable guidance throughout. I also thank Jim Bryan and Georg Oberdieck for pointing out the quasi-modularity in the initial results, Izzet Coskun for encouraging me to work out the genus 3 case, and Dawei Chen, Raymond Cheng, Bong Lian, Nicola Pagani, Renata Picciotto, Michael Thaddeus, and the anonymous referees for helpful comments on drafts of this work. This work was completed in part with the support of an NSF Graduate Research Fellowship.


\section{Preliminaries}\label{prelim}

\subsection{Conventions}\label{conventions}

We work over $\bC$. Fiber products are over $\Spec(\bC)$ unless otherwise stated. All curves, unless otherwise stated, are assumed projective and connected with only nodes as singularities. The \textit{genus} of a curve $X$ refers to its arithmetic genus and is denoted $p_a(X)$. A \textit{rational} curve is an irreducible curve of geometric genus 0. All moduli spaces are understood to be moduli stacks, rather than coarse spaces. In all figures, unlabelled irreducible components of curves are rational, and all other components are labelled with their geometric genus.

If $X$ is a nodal curve, its \textit{stabilization}, obtained by contracting rational tails and bridges (that is, non-stable components), is denoted $X^s$. We use similar notation for pointed nodal curves.

All Chow rings are taken with rational coefficients and are denoted $A^{*}(X)$, where $X$ is a variety or Deligne-Mumford stack over $\bC$. When referring to Chow groups, we use subscripts (recording the dimensions of cycles) and superscripts (recording their codimensions) interchangeably when $X$ is smooth. We will frequently refer to the Chow class of a proper and generically finite morphism $f:Y\to X$, by which we mean $f_{*}([Y])$. The class of $f$ in $A^{*}(X)$ is denoted $[f]$. When there is no opportunity for confusion, we sometimes refer to the same class by ``the class of $Y$'' or $[Y]\in A^{*}(X)$.'' If $X$ is proper and $f:X\to\Spec(\bC)$ is the structure morphism, we denote the proper pushforward map $f_{*}$ by $\int_X$.

We deal throughout this paper with boundary classes on moduli spaces of curves. We will find it more convenient to carry out intersection-theoretic calculations using classes obtained as pushforwards of fundamental classes from (products of) moduli spaces of curves of lower genus, and label these classes using upper case Greek letters. For example, when $g\ge2$, we denote by $\Delta_0\in A^{*}(\barM_g)$ the class of the morphism $\barM_{g-1,2}\to\barM_{g}$ that glues together the two marked points. We reserve lower-case letters for substack classes (also known as \textit{$Q$-classes}): for example, $\delta_0\in A^{*}(\barM_g)$ is the class of the substack of curves with a non-separating node. We have
\begin{equation*}
\delta_0=\frac{1}{2}\Delta_0.
\end{equation*}

In general, the denominator is the order of the automorphism group of the stable graph associated to the boundary stratum, which in this case is the graph consisting of a single vertex and a self-loop.

\subsection{Intersection numbers on moduli spaces of curves}\label{int_numbers_mg}

Here, we collect notation for various classes on moduli spaces of curves, and intersection numbers of these classes. We will frequently abuse notation: for instance, $\Delta_0$ will always denote the class of the locus of irreducible nodal curves on any $\barM_{g,n}$, but the spaces on which these classes are defined will be clear from the context. The intersection numbers given here can be verified using the {\tt admcycles.sage} package, \cite{sage}.

\subsubsection{$\barM_{1,2}$}\label{int_numbers_m12}

The rational Picard group $A^1(\barM_{1,n})$ is freely generated by boundary divisors, see \cite{ac}. When $n=2$, we have the boundary divisors $\Delta_0$, parametrizing irreducible nodal curves, and $\Delta_1$, parametrizing reducible curves, see Figure \ref{Fig:classes_M12.png}.
\begin{figure}
     \includegraphics[width=.75\linewidth]{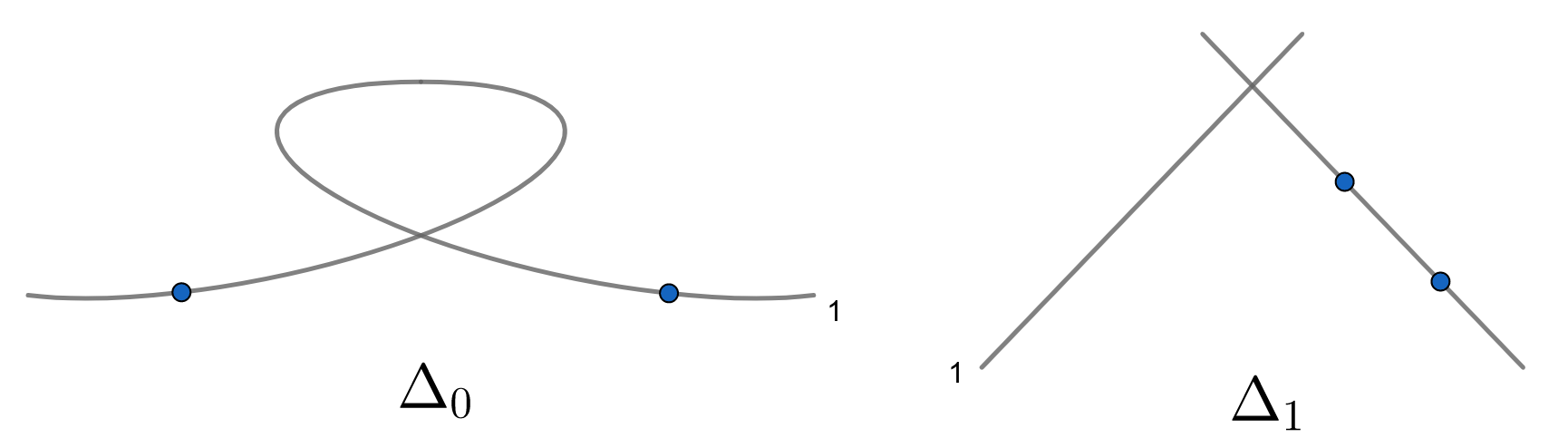}
     \caption{Boundary classes in $A^1(\barM_{1,2})$}\label{Fig:classes_M12.png}
\end{figure}
The intersection pairing is as follows:
\begin{center}
\begin{tabular}{l | c c}
 & $\Delta_0$ & $\Delta_1$\\
 \hline
$\Delta_0$ & $0$ & $1$\\
$\Delta_1$ & $1$ & $-\frac{1}{24}$ \\
\end{tabular}
\end{center}

\subsubsection{$\barM_{1,3}$}\label{int_numbers_m13}

In $A^1(\barM_{1,3})$, we have the boundary divisor $\Delta_0$ parametrizing irreducible nodal curves, and the boundary divisors $\Delta_{1,S}$, where $S\subset\{1,2,3\}$, parametrizing reducible nodal curves with the marked points corresponding to elements of $S$ lying on the rational component, see Figure \ref{Fig:classes_M13_dim2}. (We require $|S|\ge2$.)

For the same $S$, let $\Delta_{01,S}\in A^2(\barM_{1,3})$ be the class of curves in the boundary divisor $\Delta_{1,S}$ whose genus 1 component is nodal; if $|S|=2$, let $\Delta_{11,S}\in A^2(\barM_{1,3})$ be the class of curves consisting of a chain of three components, where the rational tail contains the two marked points corresponding to the elements of $S$, see Figure \ref{Fig:classes_M13_dim1}.

\begin{figure}[!htb]
     \includegraphics[width=.7\linewidth]{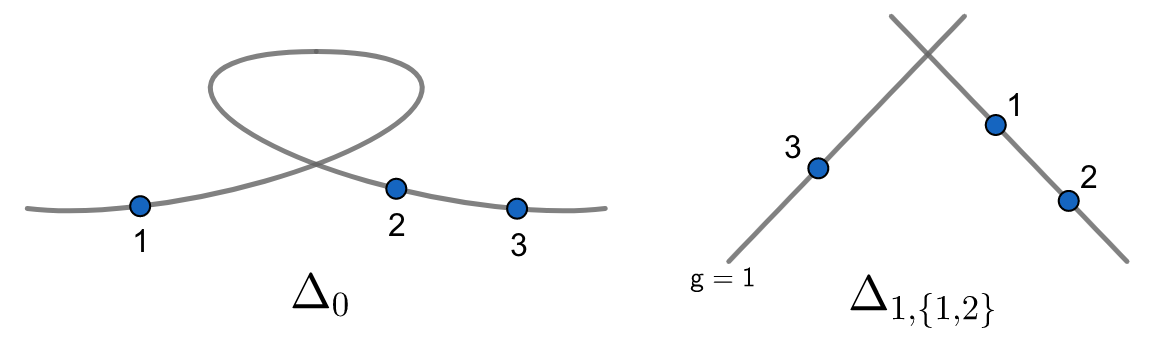}
     \caption{Some boundary classes in $A^1(\barM_{1,3})$}\label{Fig:classes_M13_dim2}
\end{figure}

\begin{figure}[!htb]
     \includegraphics[width=.7\linewidth]{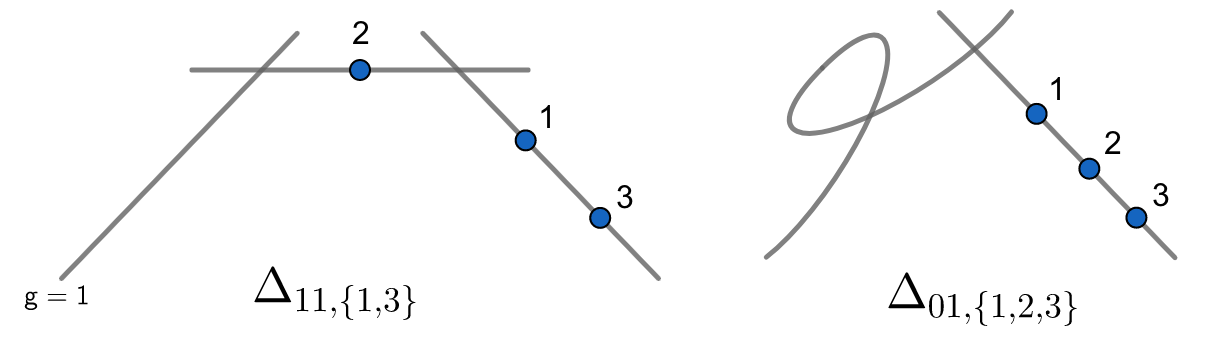}
     \caption{Some boundary classes in $A^2(\barM_{1,3})$}\label{Fig:classes_M13_dim1}
\end{figure}

We will not need the intersection numbers of the boundary divisors with all curve classes in $A^2(\barM_{1,3})$, but we record the intersections of the boundary divisors with those defined above:

\begin{center}
\begin{tabular}{l | c c c c c}
 & $\Delta_0$ & $\Delta_{1,\{1,2,3\}}$ & $\Delta_{1,\{2,3\}}$ & $\Delta_{1,\{1,3\}}$ & $\Delta_{1.\{1,2\}}$\\
 \hline
$\Delta_{01,\{1,2\}}$ & 0 & $1$ & 0 & 0 & $-1$ \\
$\Delta_{01,\{1,3\}}$ & 0 & $1$ & 0 & $-1$ & 0\\
$\Delta_{11,\{1,2\}}$ & $1$ & $-\frac{1}{24}$ & 0 & 0 & 0\\
$\Delta_{11,\{1,3\}}$ & $1$ & $-\frac{1}{24}$ & 0 & 0 & 0\\
\end{tabular}
\end{center}

\subsubsection{$\barM_2$}\label{int_numbers_m2}

Mumford has computed $A^{*}(\barM_2)$ in \cite{mumford}. $A^1(\barM_2)$ is generated by the boundary classes $\Delta_0,\Delta_1$, parametrizing irreducible nodal curves and reducible curves, respectively, see Figure \ref{Fig:classes_M2_dim2}. $A^2(\barM_2)$ is generated by the boundary classes $\Delta_{00},\Delta_{01}$, parametrizing irreducible binodal curves and reducible curves where one component is a rational nodal curve, respectively, see Figure \ref{Fig:classes_M2_dim1}.

\begin{figure}[!htb]
     \includegraphics[width=.75\linewidth]{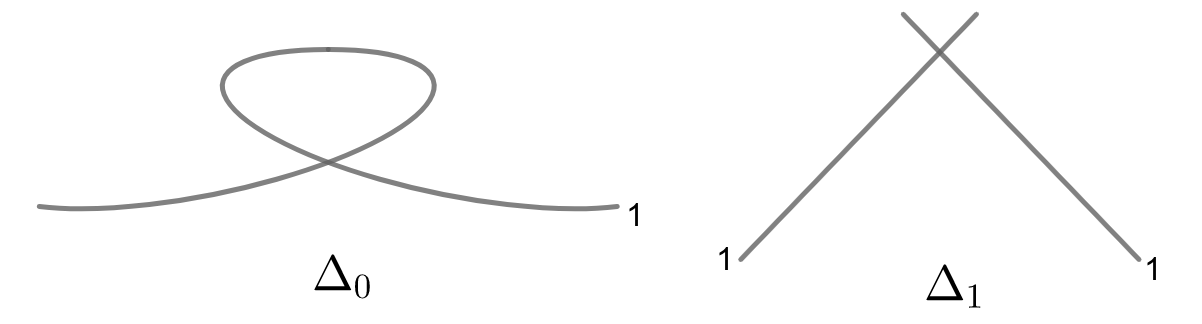}
     \caption{Boundary classes in $A^1(\barM_2)$}\label{Fig:classes_M2_dim2}
\end{figure}

\begin{figure}[!htb]
     \includegraphics[width=.75\linewidth]{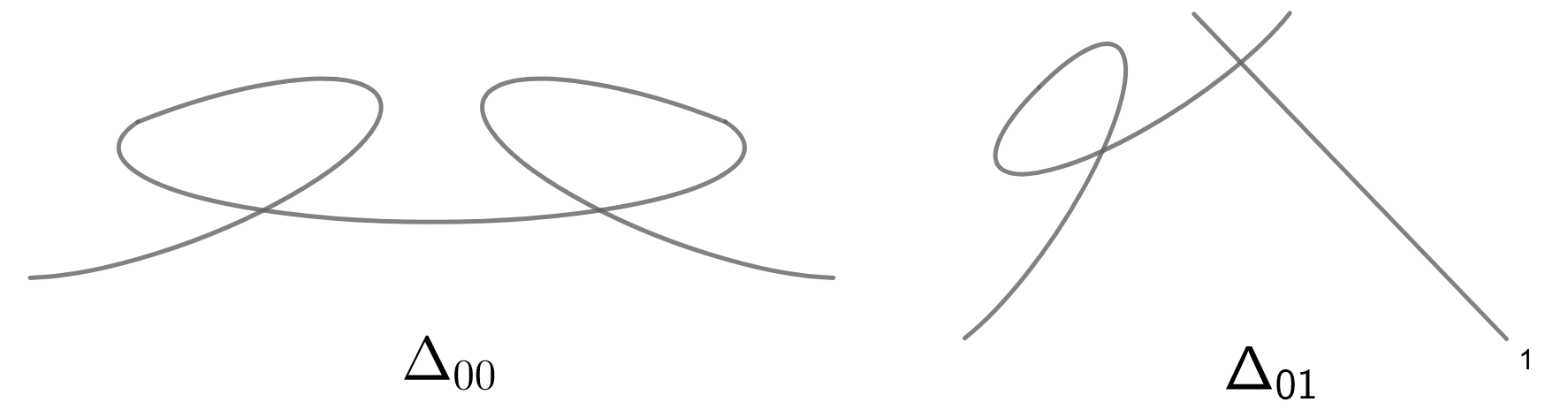}
     \caption{Boundary classes in $A^2(\barM_2)$}\label{Fig:classes_M2_dim1}
\end{figure}

The intersection pairing is as follows:
\begin{center}
\begin{tabular}{l | c c}
 & $\Delta_{0}$ & $\Delta_{1}$\\
 \hline
$\Delta_{00}$ & $-4$ & $2$\\
$\Delta_{01}$ & $1$ & $-\frac{1}{12}$ \\
\end{tabular}
\end{center}

\subsubsection{$\barM_{2,1}$}\label{int_numbers_m21}

It follows from \cite{faber_thesis} that $A^2(\barM_{2,1})$ has dimension 5, with a basis given by the boundary classes $\Delta_{00},\Delta_{01a},\Delta_{01b},\Xi_1,\Delta_{11}$, shown in Figure \ref{Fig:classes_M21_dim2}.

\begin{figure}
     \includegraphics[width=.85\linewidth]{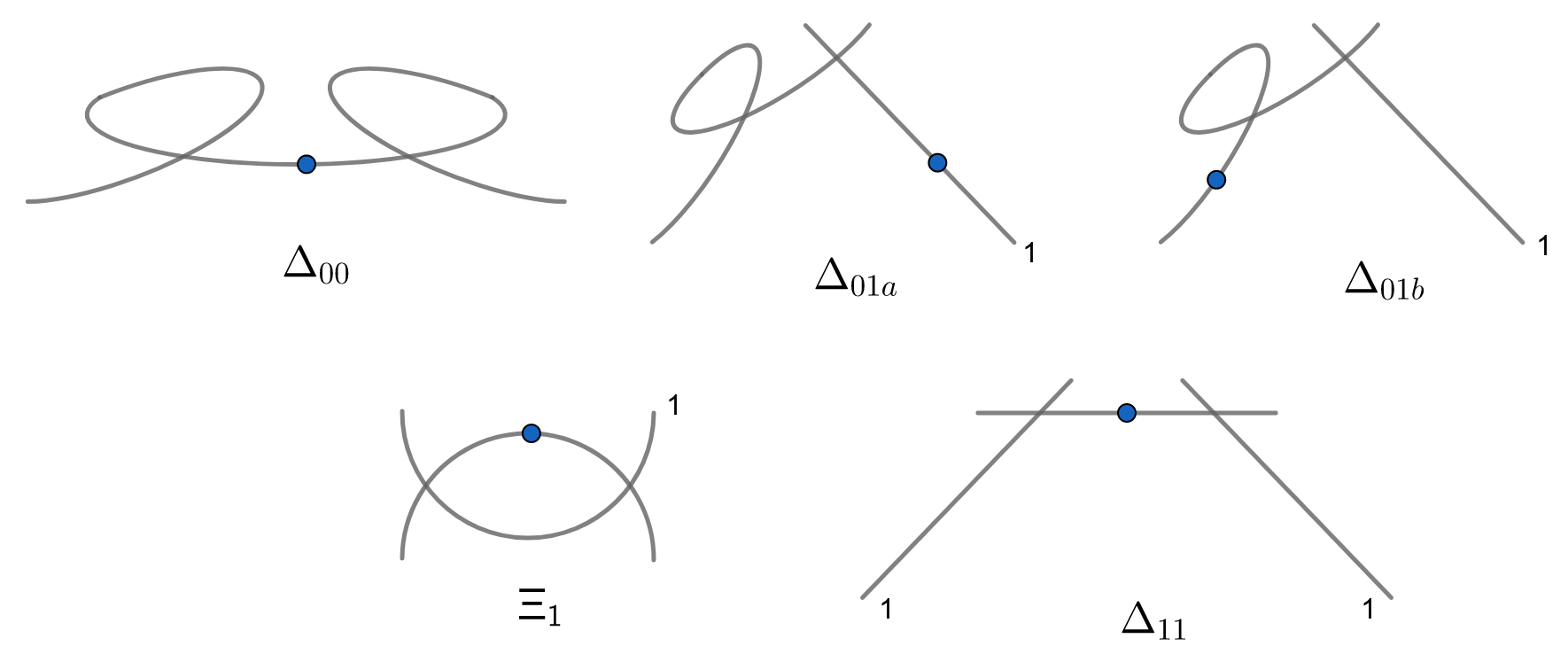}
     \caption{Boundary classes in $A^2(\barM_{2,1})$}\label{Fig:classes_M21_dim2}
\end{figure}

The intersection pairing is as follows:

\begin{center}
\begin{tabular}{l | c c c c c}
 & $\Delta_{00}$ & $\Delta_{01a}$ & $\Delta_{01b}$ & $\Xi_1$ & $\Delta_{11}$ \\
 \hline
$\Delta_{00}$ & 0 & 0 & 0 & $-4$ & 2\\
$\Delta_{01a}$ & 0 & 1 & $-1$ & 1 & 0\\
$\Delta_{01b}$ & 0 & $-1$ & 1 & 0 & $-\frac{1}{12}$ \\
$\Xi_1$ & $-4$ & 1 & 0 & $\frac{1}{12}$ & 0 \\
$\Delta_{11}$ & 2 & 0 & $-\frac{1}{12}$ & 0 & $\frac{1}{288}$
\end{tabular}
\end{center}

We will also need the classes $\Delta_1\in A^1(\barM_{2,1})$ and $\Gamma_{(5)},\Gamma_{(6)},\Gamma_{(11)}\in A^3(\barM_{2,1})$, shown in Figure \ref{Fig:classes_M21_dim13}.
\begin{figure}
     \includegraphics[width=.8\linewidth]{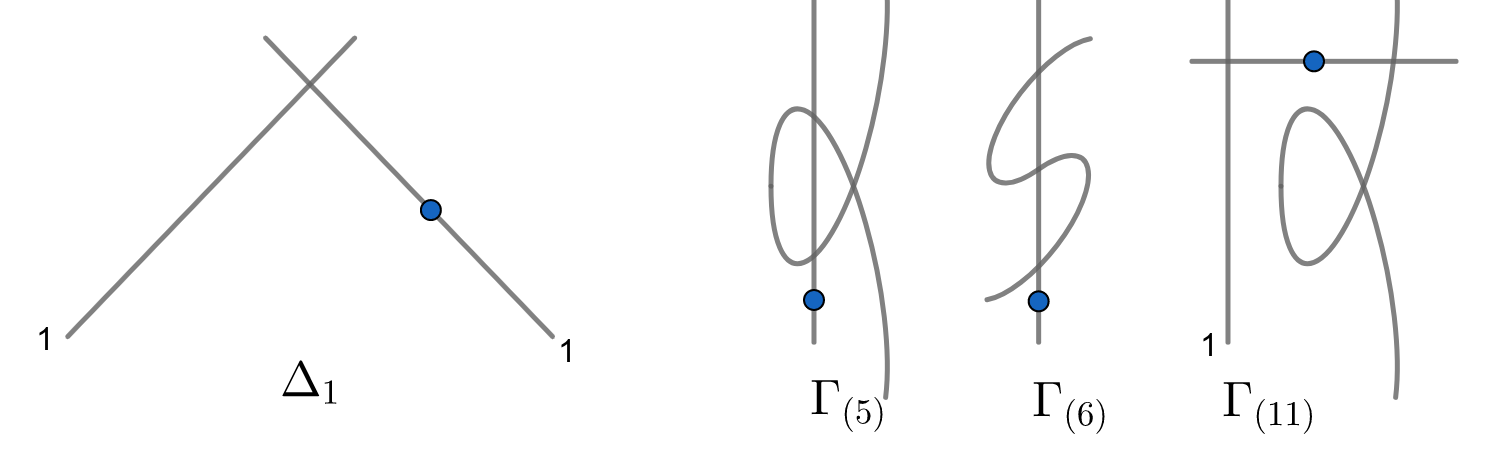}
     \caption{Some boundary classes in $A^1(\barM_{2,1})$ and $A^3(\barM_{2,1})$}\label{Fig:classes_M21_dim13}
\end{figure}

The subscripts in the classes $\Gamma_{(i)}\in A^3(\barM_{2,1})$ are chosen in such a way that $\Gamma_{(i)}\times\barM_{1,1}=\Delta_{[i]}\in A^4(\barM_{3})$, see \S\ref{int_numbers_m3}. We have the following intersection numbers:

\begin{center}
\begin{tabular}{l | c c c}
 & $\Gamma_{(5)}$ & $\Gamma_{(6)}$ & $\Gamma_{(11)}$ \\
 \hline
$\Delta_{1}$ & 1 & 0 & $-\frac{1}{24}$\\
\end{tabular}
\end{center}

\subsubsection{$\barM_{3}$}\label{int_numbers_m3}

Faber has computed $A^{*}(\barM_3)$ in \cite{faber_thesis}. We have that $A^2(\barM_3)$ and $A^4(\barM_3)$ both have dimension 7 and pair perfectly. To describe bases of these groups, we first recall the definitions of the $\lambda$ and $\kappa$ classes. Let $u:\overline{\cC}_g\to\barM_g$ be the universal curve, let $\omega_{\overline{\cC}_g/\barM_g}$ be the relative dualizing sheaf, and let $K\in A^{1}(\overline{\cC}_g)$ denote the divisor class of $\omega_{\overline{\cC}_g/\barM_g}$. Then, by definition:
\begin{align*}
\lambda_i&=c_i(u_{*}\omega_{\overline{\cC}_g/\barM_g})\in A^{i}(\barM_g)\\
\kappa_i&=u_{*}(K^{i+1})\in A^{i}(\barM_g)
\end{align*}

We will only need the class $\lambda_1$ in this paper, so we write $\lambda=\lambda_1$ in $A^1(\barM_3)$, with no risk of confusion. Then, a basis for $A^2(\barM_3)$ is given by the seven classes $\lambda^2,\lambda\delta_0,\lambda\delta_1,\delta_0^2,\delta_0\delta_1,\delta_1^2,\kappa_2$.

A basis for $A^4(\barM_3)$ is given by surface classes $\Delta_{[i]}$ for $i\in\{1,4,5,6,8,10,11\}$, retaining the indexing from \cite{faber_thesis}, see Figure \ref{Fig:classes_M3_dim2}.
\begin{figure}
     \includegraphics[width=.92\linewidth]{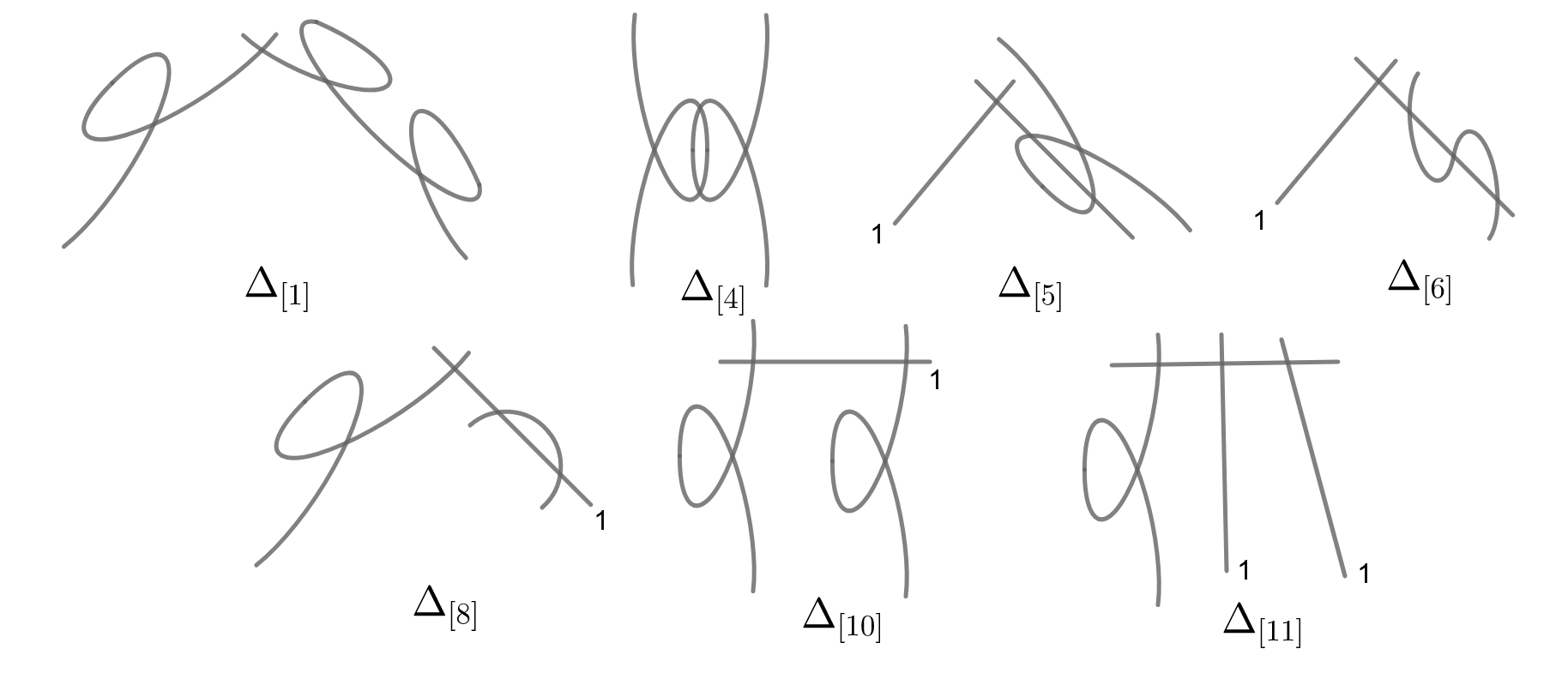}
     \caption{Boundary classes in $A^4(\barM_{3})$}\label{Fig:classes_M3_dim2}
\end{figure}

We have the following intersection numbers:

\begin{center}
\begin{tabular}{l | c c c c c c c}
 & $\lambda^2$ & $\lambda\delta_0$ & $\lambda\delta_1$ & $\delta_0^2$ & $\delta_0\delta_1$ & $\delta_1^2$ & $\kappa_2$\\
 \hline
$\Delta_{[1]}$ & 0 & 0 & 0 & 0 & 4 & $-3$ & 1\\
$\Delta_{[4]}$ & 0 & 0 & 0 & 8 & $-4$ & 2 & 0\\
$\Delta_{[5]}$ & 0 & $-\frac{1}{12}$ & $\frac{1}{24}$ & $-2$ & $\frac{7}{12}$ & $-\frac{1}{12}$ & 0\\
$\Delta_{[6]}$ & 0 & 0 & $-\frac{1}{24}$ & 0 & $-\frac{1}{2}$ & $\frac{1}{12}$ & 0\\
$\Delta_{[8]}$ & 0 & $-\frac{1}{12}$ & $\frac{1}{24}$ & $-\frac{11}{6}$ & $\frac{1}{2}$ & $-\frac{1}{24}$ & $\frac{1}{24}$\\
$\Delta_{[10]}$ & 0 & 0 & $-\frac{1}{24}$ & 0 & $-\frac{1}{2}$ & $\frac{1}{8}$ & $\frac{1}{24}$\\
$\Delta_{[11]}$ & $\frac{1}{288}$ & $\frac{1}{24}$ & $-\frac{1}{288}$ & $\frac{1}{2}$ & $-\frac{1}{24}$ & $\frac{1}{288}$ & 0\\
\end{tabular}
\end{center}

\subsection{Admissible covers}\label{admissible_covers_prelim}
We recall the definition of \cite{hm}:

\begin{defn}\label{admissible_covers_def}
Let $X,Y$ be nodal curves. Let $b=(2p_a(X)-2)-d(2p_a(Y)-2)$, and let $y_1,\ldots,y_b\in Y$ be such that $(Y,y_1,\ldots,y_b)$ is stable. Then, an \textbf{admissible cover} consists of the data of the stable marked curve $(Y,y_1,\ldots,y_b)$ and a finite morphism $f:X\to Y$ such that:
\begin{itemize}
\item $f(x)$ is a smooth point of $Y$ if and only if $x$ is a smooth point of $X$,
\item $f$ is simply branched over the $y_i$ and \'{e}tale over the rest of the smooth locus of $Y$, and
\item at each node of $X$, the ramification indices of $f$ restricted to the two branches are equal.
\end{itemize}
\end{defn}

\begin{rem}\label{preimage_adm_smooth}
It is clear that non-separating nodes of $X$ must map to non-separating nodes of $Y$, and self-nodes of $X$ (nodes at which both branches belong to the same component) map to self-nodes of $Y$. Hence, the pre-image of a smooth component of $Y$ must be a disjoint union of smooth components of $X$. 
\end{rem}

Isomorphism classes of admissible covers of degree $d$ from a genus $g$ curve to a genus $h$ curve are parametrized by a proper Deligne-Mumford stack $\Adm_{g/h,d}$, see \cite{hm,mochizuki,acv}. $\Adm_{g/h,d}$ contains the Hurwitz space $\cH_{g/h,d}$ parametrizing simply branched covers of smooth curves as a dense open substack. 

Let $b=2g-2-d(2h-2)$. Then, we have a forgetful map $\psi_{g/h,d}:\Adm_{g/h,d}\to\barM_{h,b}$ remembering the (marked) target, and another $\pi_{g/h,d}:\Adm_{g/h,d}\to\barM_{g,b}$ remembering the stabilization of the source, marked by the ramification points $x_1,\ldots,x_b$ mapping to $y_1,\ldots,y_b$. We will often abuse notation and write $\pi_{g/h,d}:\Adm_{g/h,d}\to\barM_{g,r}$ for $r<b$, obtained by post-composing with the map $\barM_{g,b}\to\barM_{g,r}$ forgetting the last $b-r$ points.

We will also allow the case $g=h=1$, where we agree that $\Adm_{1/1,d}$ parametrizes \textit{pointed} admissible covers of a \textit{marked} stable genus 1 curve $(Y,y_1)$ (i.e. a point of $\barM_{1,1}$) by a marked nodal curve $(X,x_1)$, also of genus 1. (We have in this case that $b=0$, so the target curve cannot be stable unless we include the data of the marked point.)

\begin{lem}\label{adm_to_mh_quasifinite}
The morphism $\psi_{g/h,d}:\Adm_{g/h,d}\to\barM_{h,b}$ is quasifinite (and hence finite).
\end{lem}

\begin{proof}
Over the open locus $\cM_{h,b}$, this is classical: the number of points in any fiber is given by a Hurwitz number, counting isomorphism classes of monodromy actions of the finitely generated group $\pi_1(Y-\{y_1,\ldots,y_b\})$ on a general fiber of $f:X\to Y$. Over a general point of any boundary stratum of $\barM_{h,b}$ parametrizing admissible covers $f:X\to Y$, there are a finite number of possible collections of ramification profiles above the nodes of $Y$, each of which leads to finitely many collections of covers of the individual components of $Y$, which in turn can be glued together in finitely many ways.
\end{proof}

We also recall from \cite{hm} the explicit local description of $\Adm_{g/h,d}$. Let $[f:X\to Y]$ be a point of $\Adm_{g/h,d}$. Let $y'_1,\ldots,y'_n$ be the nodes of $Y$, and let $y_{1},\ldots,y_{b}\in Y$ be the branch points of $f$. Let $\mathbb{C}[[t_1,\ldots,t_{3h-3+b}]]$ be the coordinate ring of the deformation space of $(Y,y_1,\ldots,y_b)$, so that $t_1,\ldots,t_n$ are smoothing parameters for the nodes $y'_1,\ldots,y'_n$. Let $x_{i,1},\ldots,x_{i,r_i}$ be the nodes of $X$ mapping to $y'_i$, and denote the ramification index of $f$ at $x_{i,j}$ by $a_{i,j}$. 

\begin{prop}[\cite{hm}]\label{adm_defos}
The complete local ring of $\Adm_{g/h,d}$ at $[f]$ is
\begin{equation*}
\mathbb{C}\left[\left[t_1,\ldots,t_{3h-3+b},\{t_{i,j}\}^{1\le i\le n}_{1\le j\le r_i}\right]\right]/\left(t_1=t_{1,1}^{a_{1,1}}=\cdots=t_{1,r_1}^{a_{1,r_1}},\ldots,t_n=t_{n,1}^{a_{n,1}}=\cdots=t_{n,r_n}^{a_{n,r_n}}\right).
\end{equation*} 
\end{prop}

Here, the variable $t_{i,j}$ is the smoothing parameter for $X$ at $x_{i,j}$. In particular, $\Adm_{g/h,d}$ is Cohen-Macaulay of pure dimension $3h-3+b$. Moreover, if the $a_{i,j}$ are all equal to 1, that is, $f$ is unramified over the nodes of $Y$ (or if $Y$ is smooth to begin with), then $\Adm_{g/h,d}$ is smooth at $[f]$.

One can readily extend the theory to construct stacks of admissible covers with arbitrary ramification profiles; we use this in \S\ref{dd22_pencils}. Even in this more general setting, we always require the target curve, marked with branch points, to be stable.

In this paper, we primarily study the case $h=1$, that is, the moduli of covers of elliptic curves. The space $\Adm_{g/1,d}$ is reducible when $d$ is composite, due to the existence of covers $C\to E_1\to E_2$ factoring through a non-trivial isogeny. However, the open and closed substack $\Adm_{g/1,d}^{\text{prim}}$ parametrizing \emph{primitive} covers, that is, those that do not factor through a non-trivial isogeny (more generally, through a non-trivial admissible cover of genus 1 curves), is irreducible. In fact, this is already true for a fixed elliptic target, see \cite{gh} or \cite[Theorem 1.4]{bujokas}. For our enumerative results, however, the individual components of $\Adm_{g/1,d}$ play no essential role: we consider the entire moduli space, including the components parametrizing non-primitive covers.

\subsection{Quasimodular forms}\label{qmod_prelim}

A possible reference for this section is \cite{royer}. For positive integers $d,k$, define
\begin{equation*}
\sigma_k(d)=\sum_{a|d}a^k
\end{equation*}
Recall that the ring $\Qmod$ of quasimodular forms is generated over $\bC$ by the Eisenstein series
\begin{align*}
E_2&=1-24\sum_{d=1}^{\infty}\sigma_1(d)q^d\\
E_4&=1+240\sum_{d=1}^{\infty}\sigma_3(d)q^d\\
E_6&=1-504\sum_{d=1}^{\infty}\sigma_5(d)q^d,
\end{align*}
where we take $q$ to be a formal variable, and that more generally, for integers $k\ge1$, we have
\begin{equation*}
E_{2k}=1-\frac{4k}{B_{2k}}\sum_{n=1}^{\infty}\sigma_{2k-1}(n)q^n\in\Qmod,
\end{equation*}
where $B_{2k}$ is a Bernoulli number. The \emph{weight} of $E_{2k}$ is ${2k}$, and $\Qmod$ is a graded $\bC$-algebra by weight. 

We have the Ramanujan identities
\begin{align*}
q\frac{dE_2}{dq}&=\frac{E_2^2-E_4}{12}\\
q\frac{dE_4}{dq}&=\frac{E_2E_4-E_6}{3}\\
q\frac{dE_6}{dq}&=\frac{E_2E_6-E_4^2}{2},
\end{align*}
so in particular
\begin{equation*}
\sum_{d=1}^{\infty}P(d)\sigma_{2k-1}(d)q^d\in\Qmod
\end{equation*}
for any $P(d)\in\bC[d]$. Thus, Theorem \ref{modularity_statement} will be an immediate consequence of Theorems \ref{main_thm_g2} and \ref{main_thm_g3}.

The Ramanujan identities also give the convolution formulas
\begin{align*}
\sum_{d_1+d_2=d}\sigma_1(d_1)\sigma_1(d_2)&=\left(-\frac{1}{2}d+\frac{1}{12}\right)\sigma_1(d)+\frac{5}{12}\sigma_3(d)\\
\sum_{d_1+d_2=d}d_1\sigma_1(d_1)\sigma_1(d_2)&=\left(-\frac{1}{4}d^2+\frac{1}{24}d\right)\sigma_1(d)+\frac{5}{24}d\sigma_3(d)\\
\sum_{d_1+d_2+d_3=d}\sigma_1(d_1)\sigma_1(d_2)\sigma_1(d_3)&=\left(\frac{1}{8}d^2-\frac{1}{16}d+\frac{1}{192}\right)\sigma_1(d)+\left(-\frac{5}{32}d+\frac{5}{96}\right)\sigma_3(d)+\frac{7}{192}\sigma_5(d)
\end{align*}


\section{Auxiliary computations}\label{prelude}

In this section, we record a number of enumerative results for branched covers that we will take as inputs in the main computation.

\subsection{Counting isogenies}\label{m11_d-ell}

\begin{lem}\label{count_isogenies}
Let $(E,p)$ be an elliptic curve and $d$ be a positive integer. Then, the number of isomorphism classes of isogenies $E\to F$ of degree $d$ is $\sigma_1(d)$. Likewise, the number of isomorphism classes of isogenies $F\to E$ of degree $d$ is $\sigma_1(d)$.
\end{lem}

\begin{proof}
We see that these two numbers are equal by taking duals, so it suffices to count isogenies $E\to F$ of degree $d$, i.e., quotients of $E$ by a subgroup of order $d$, which is the number of index $d$ sublattices of $\bZ^2$. A sublattice of $\bZ^2$ is determined by a $\bZ$-basis $(a,0),(b,c)$, where $a,c$ are positive and as small as possible; $b$ is uniquely determined modulo $a$. As $ac=d$, the number of such sublattices is exactly $\sigma_1(d)$.
\end{proof}

\begin{cor}\label{adm_m11_m11}
The degrees of the morphisms $\psi_{1/1,d}:\Adm_{1/1,d}\to\barM_{1,1}$ and $\pi_{1/1,d}:\Adm_{1/1,d}\to\barM_{1,1}$ remembering the target and (contracted) source, respectively, of a cover, are both $\sigma_1(d)$. Moreover, both morphisms are unramified over $\cM_{1,1}$.
\end{cor}

\begin{proof}
The first statement is exactly the content of Lemma \ref{count_isogenies}. To see that both morphisms are unramified over $\cM_{1,1}$, note that the open locus $\cH_{1/1,d}\subset\Adm_{1/1,d}$ parametrizing covers of smooth curves is smooth, and that the set-theoretic fibers of both morphisms over any point of $\cM_{1,1}$ all have the same size.
\end{proof}

\subsection{The 2-pointed $d$-elliptic locus on $\barM_{1,2}$}\label{m12_d-ell}

\begin{lem}\label{count_pointed_isogenies}
Let $(E,p)$ be an elliptic curve and $d$ a positive integer. Then, the number of pairs (up to isomorphism) $(f,q)$ where $f:E\to F$ is an isogeny and $q\neq p$ is a pre-image of the origin of $F$ is $(d-1)\sigma_1(d)$.
\end{lem}

\begin{proof}
We give a bijection between the set of $(f,q)$ and the set of pairs $(G,g)$ where $G\subset E[d]$ is a subgroup of order $d$ and $0\neq g\in G$. The claim then follows from Lemma \ref{count_isogenies}. In one direction, given $(f,q)$, we take $G=\ker(f)$ and $g=q$. In the other, let $F=E/G$ and $f$ be the quotient map, and take $q=g$.
\end{proof}

Let $\Adm_{1/1,d,2}$ be the moduli space of triples $(f,x_1,x_2)$, where $f:X\to Y$ is a degree $d$ cover of a marked genus 1 curve $(Y,y)$ by a genus 1 curve $X$, and $x_1,x_2\in X$ are distinct points with $f(x_1)=f(x_2)=y$. We have a finite morphism $\psi_{1/1,d,2}:\Adm_{1/1,d,2}\to\barM_{1,1}$ remembering the target curve, and a morphism $\pi_{1/1,d,2}:\Adm_{1/1,d,2}\to\barM_{1,2}$ remembering the stabilized source curve.

\begin{prop}\label{m12_d-ell_class}
Adopting the notation of \S\ref{int_numbers_m12}, we have:
\begin{align*} 
\int_{\barM_{1,2}}[\pi_{1/1,d,2}]\cdot\Delta_0&=(d-1)\sigma_1(d),\\
\int_{\barM_{1,2}}[\pi_{1/1,d,2}]\cdot\Delta_1&=0.
\end{align*}
\end{prop}

\begin{proof}
Let $(E,p)$ be a general elliptic curve. Any two geometric points of $\barM_{1,1}$ are equivalent, so the boundary class $\Delta_0$ is equivalent to the class of the morphism $t_E:E\to\barM_{1,2}$ sending $q\mapsto (E,p,q)$. Then, the first statement follows from Lemma \ref{count_pointed_isogenies} provided the intersection of $t_E$ and $\pi_{1/1,d,2}$ is transverse. This is easy to see: at an intersection point $(E,p,q)$, a tangent vector from $E$ fixes $(E,p)$ but moves $q$ to first order, while a tangent vector from $\Adm_{1/1,d,2}$ moves $(E,p)$ to first order, owing to Corollary \ref{adm_m11_m11}.

The second statement follows from the fact that no (pointed) admissible cover $f:X\to Y$ in $\Adm_{1/1,d,2}$ has the property that $X$ contracts to a curve in $\Delta_1$. Indeed, if $Y$ is singular, then it must be a nodal cubic. Then, one checks that $X$ must be a cycle of $m$ rational curves, each of which maps to the normalization of $Y$ via the map $x\mapsto x^{d/m}$, totally ramified at the nodes. (To see this, one can follow the method of 
\S\ref{adm_g2_delta0}.) The contraction of such a curve does not lie in $\Delta_1$.
\end{proof}

\begin{cor}\label{m12_d-ell_class_formula}
We have
\begin{equation*}
[\pi_{1/1,d,2}]=(d-1)\sigma_1(d)\left(\frac{1}{24}\Delta_0+\Delta_1\right)
\end{equation*}
in $A^1(\barM_{1,2})$.
\end{cor}

\begin{proof}
Immediate from \S\ref{int_numbers_m12}.
\end{proof}

\subsection{Doubly totally ramified covers of $\bP^1$}\label{dd22_pencils}

The following is an easy special case of \cite[Theorem 1.3]{lian} or \cite[Corollary 4.14]{lian} .

\begin{lem}\label{count_dd22}
Let $(E,x_1)$ be a general elliptic curve and $d$ a positive integer. Then, the number of tuples $(f,x_1,x_2,x_3,x_4)$, where $x_1,\ldots,x_4\in E$ are distinct and $f:E\to\bP^1$ is a degree $d$ morphism (considered up to automorphism of the target) totally ramified at $x_1,x_2$ and simply ramified at $x_3,x_4$, is $2(d^2-1)$.
\end{lem}

\begin{proof}
The linear system defining $f$ must be a 2-dimensional subspace $W$ of $V=H^0(E,\cO(d\cdot x_1))$. In order for $W$ to be totally ramified at $x_1$, we need $\cO(d\cdot x_1)\cong\cO(d\cdot x_2)$, that is, $x_1\in E[d]-\{x_1\}$. For such an $x_2$, there are unique (up to scaling) sections in $V$ vanishing to maximal order at $x_1,x_2$; thus $W$ is uniquely determined by the $d^2-1$ possible choices of $x_2$. Moreover, $f$ will be simply branched over two distinct points $x_3,x_4$ of $\bP^1$ unless it has two simple ramification points over the same point of $\bP^1$ or a triple ramification point; however, this will only happen of $E$ admits a degree $d$ cover of $\bP^1$ branched over 3 points, which is impossible for $E$ general. There are two ways to label the simple ramification points, so the conclusion follows.
\end{proof}

\begin{cor}\label{count_dd22_translate}
Let $(E,x_1)$ be a general elliptic curve and $d$ a positive integer. Then, the number of tuples $(f,x_1,x_2,x_3,x_4)$, where $x_1,\ldots,x_4\in E$ are distinct and $f:E\to\bP^1$ is a degree $d$ morphism simply ramified at $x_1,x_2$ and totally ramified at $x_3,x_4$, is $2(d^2-1)$.
\end{cor}

\begin{proof}
Pullback by translation by $x_3-x_1$ defines a bijection with the objects counted here and those in Lemma \ref{count_dd22}.
\end{proof}

Let $\Adm_{1/0,d}^{d,d,2,2}$ be the moduli space of degree $d$ admissible covers $f:X\to Y$, where $X$ has genus 1, $Y$ has genus 0, and $f$ is ramified at four points $x_1,x_2,x_3,x_4$ to orders $d,d,2,2$, respectively. We consider the map $\pi_{1/0,d}^{d,d,2,2}:\Adm_{1/0,d}^{d,d,2,2}\to\barM_{1,3}$ sending $f$ to the stabilization of $(X,x_1,x_2,x_3)$.

\begin{prop}\label{dd22_class_m13}
Adopting the notation of \S\ref{int_numbers_m13}, we have:
\begin{align*}
\int_{\barM_{1,3}}[\pi_{1/0,d}^{d,d,2,2}]\cdot\Delta_0&=2(d^2-1),\\
\int_{\barM_{1,3}}[\pi_{1/0,d}^{d,d,2,2}]\cdot\Delta_{1,S}&=0,
\end{align*}
for all $S\subset\{1,2,3\}$ with $|S|\ge2$.
\end{prop}

\begin{proof}
Similarly to the proof of Proposition \ref{m12_d-ell_class}, we replace $\Delta_0$ with the class of $\overline{M}_{E,3}$, as defined by the Cartesian diagram
\begin{equation*}
\xymatrix{
\overline{M}_{E,3} \ar[r] \ar[d] & \barM_{1,3} \ar[d] \\
[E] \ar[r] & \barM_{1,1}
}
\end{equation*}
where $[E]$ is the geometric point corresponding to a general elliptic curve $E$. Then, to check that $[\overline{M}_{E,3}]$ and $\pi_{1/0,d}^{d,d,2,2}$ intersect transversely, note that $\Adm_{1/0,d}^{d,d,2,2}$ is unramified over a general point of $\cM_{1,1}$, so a tangent vector from $\Adm_{1/0,d}^{d,d,2,2}$ will deform $E$ to first order, whereas a tangent vector from $\overline{M}_{E,3}$ will fix $E$ and deform the marked points to first order. The first statement then follows from Lemma \ref{count_dd22}.

On the other hand, it is straightforward to check that the image of $\pi_{1/0,d}^{d,d,2,2}$ is disjoint from every $\Delta_{1,S}$: if $Y$ is singular and $[f:X\to Y]\in\Adm_{1/0,d}^{d,d,2,2}$, then $X$ is either a union of an elliptic curve and rational tails, all of which will be contracted, or a union of smooth rational curves. The second statement follows.
\end{proof}

\begin{prop}\label{dd2222_g3}
Let $C$ be a general curve of genus 2. Then, up to automorphisms of the target, there are $48(d^4-1)$ tuples $(f,x_1,x_2,x_3,x_4,x_5,x_6)$, where $x_1,\ldots,x_6\in C$ are distinct points, and $f:C\to\bP^1$ is a degree $d$ morphism totally ramified at $x_1,x_2$ and simply ramified at $x_3,x_4,x_5,x_6$.
\end{prop}

\begin{proof}
The fact that such an $f$ is simply ramified at four other points follows from a dimension count: the dimension of the space of covers $C\to\bP^1$ branched over 5 points or fewer is 2, so a general point of $\cM_2$ admits no such covers.

We appeal to the degeneration technique described in \cite{lian}: it suffices to count limit linear series $V$ on the reducible curve $C_0$ formed by attaching general elliptic curves $E_1,E_2$ to a copy of $\bP^1$, where each elliptic component contains three of the $x_i$. Of the $\binom{6}{3}=20$ possible distributions of the $x_i$ onto the elliptic components, there are 12 ways for $x_1,x_2$ to lie on different components, and 8 for them to lie on the same component. In the first case, it follows from Lemma \ref{count_dd22} that there are $2(d^2-1)$ possible aspects of $V$ on each $E_i$, and the aspect of $V$ on $\bP^1$ must be the unique pencil with vanishing sequence $(0,d)$ at both marked points. In the second, we have, by Corollary \ref{count_dd22_translate}, $2(d^2-1)$ possible aspects on the component containing $x_1,x_2$ and $2(2^2-1)=6$ on the other, and the aspect of $V$ on $\bP^1$ must be the unique pencil with vanishing sequences $(0,2),(d-2,d)$ at the two marked points. 

Thus, our answer is
\begin{equation*}
12\cdot(2(d^2-1))^2+8\cdot6\cdot2(d^2-1))=48(d^4-1),
\end{equation*}
as desired.
\end{proof}

\begin{rem}
One can also recover Proposition \ref{dd2222_g3} using Tarasca's formula for the closure of the locus of $[(C,p,q)]\in\barM_{2,2}$ such that $C$ admits a cover of $\bP^1$ totally ramified at $p$ and $q$, see \cite{tarasca}.
\end{rem}


\section{The $d$-elliptic locus on $\barM_2$}\label{m2_d-ell}

In this section, we prove Theorem \ref{main_thm_g2}. It suffices to compute the intersection of the morphism $\pi_{2/1,d}:\Adm_{2/1,d}\to\barM_2$ with all curve classes in $A_1(\barM_2)$. It is possible to simplify the computation by specializing to particular test curves and computing these intersection numbers with these classes directly, but instead we explain a more abstract approach that we will employ later for pointed genus 2 curves and in genus 3.

It follows from \cite{mumford} that $A_1(\cM_2)=0$, and hence that any curve class on $\barM_2$ comes from the boundary. In particular, any curve class may be represented as a rational linear combination of morphisms from a smooth, connected scheme $C$ of dimension 1 to one of the two boundary divisors:
\begin{enumerate}
\item[($\Delta_0$)] $C\to\barM_{1,2}\to\barM_{2}$
\item[($\Delta_1$)] $C\to\barM_{1,1}\times\barM_{1,1}\to\barM_{2}$
\end{enumerate}

We may furthermore take $C$ to be general, in the sense that $C$ intersects any given finite collection of subvarieties of its associated boundary divisor as generically as possible. Now, consider the intersection of such a $C$ with the admissible locus $\pi_{2/1,d}:\Adm_{2/1,d}\to\barM_2$. Note that we get a stratification of $\Adm_{2/1,d}$ by pulling back the stratification of $\barM_{1,2}$ by boundary strata under the finite morphism $\psi_{2/1,d}$, see Lemma \ref{adm_to_mh_quasifinite}. Then, $C$ may be chosen to avoid the zero-dimensional strata of $\Adm_{2/1,d}$.

Thus, when intersecting $C$ with the admissible locus, we need only consider admissible covers $f:X\to Y$ where $Y$ has at most one node; in fact, because $X$ must be singular, $Y$ must have exactly one node. We classify such covers in \S\ref{admissible_classification_g2}, then  in \S\ref{adm_g2_intersect_delta1} and \S\ref{adm_g2_intersect_delta0} compute the contributions of the covers of each topological type to the intersection numbers.

\subsection{Classification of Admissible Covers}\label{admissible_classification_g2}

Let $f:X\to Y$ be a point of $\Adm_{2/1,d}$ where $Y$ has exactly one single node, so that $[Y]$ lies in exactly one of the boundary divisors $\Delta_i$. We consider the cases $i=0,1$ separately.

\subsubsection{$[Y]\in\Delta_1$}\label{adm_g2_delta1}

Let $Y_i$ be the component of $Y$ of genus $i$, and let $y=Y_0\cap Y_1$. By assumption, both $Y_i$ are smooth. The pre-image of $Y_i$ is then a union of smooth curves; we may ignore the case where one of the components has genus 2, by the assumption on $X^s$. Thus, $f^{-1}(Y_1)$ either consists of a single genus 1 curve $X_1$, or two disjoint elliptic curves $X_1,X'_1$.

In both cases, the map $f$ must be unramified over $y$, and simply ramified over two points of $Y_0$. Thus, the pre-image of $Y_0$ consists of smooth rational curves attached to the pre-image of $Y_1$ at the pre-images of $y$, all of which map isomorphically to $Y_0$ except one, which has degree 2 over $Y_0$. In the case that $f^{-1}(Y_1)$ has two components, the degree 2 component must be a bridge between $X_1$ and $X'_1$, in order for $X$ to be connected.

We thus get covers of two topological types, which we denote by $(\Delta_0,\Delta_1)$ (Figure \ref{Fig:g2_Delta0-Delta1}) and $(\Delta_1,\Delta_1)$ (Figure \ref{Fig:g2_Delta1-Delta1}); the coordinates are the boundary strata in which $X^s$ and $Y$ lie, respectively.

\begin{figure}[!htb]
   \begin{minipage}{0.48\textwidth}
     \centering
     \includegraphics[width=.7\linewidth]{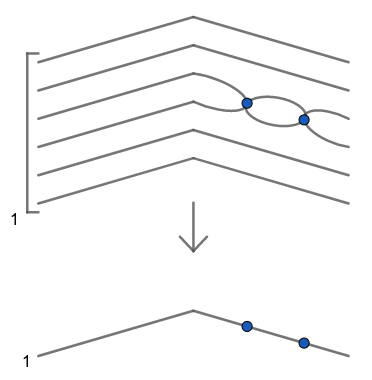}
     \caption{Cover of type $(\Delta_0,\Delta_1)$}\label{Fig:g2_Delta0-Delta1}
   \end{minipage}\hfill
   \begin{minipage}{0.48\textwidth}
     \centering
     \includegraphics[width=.7\linewidth]{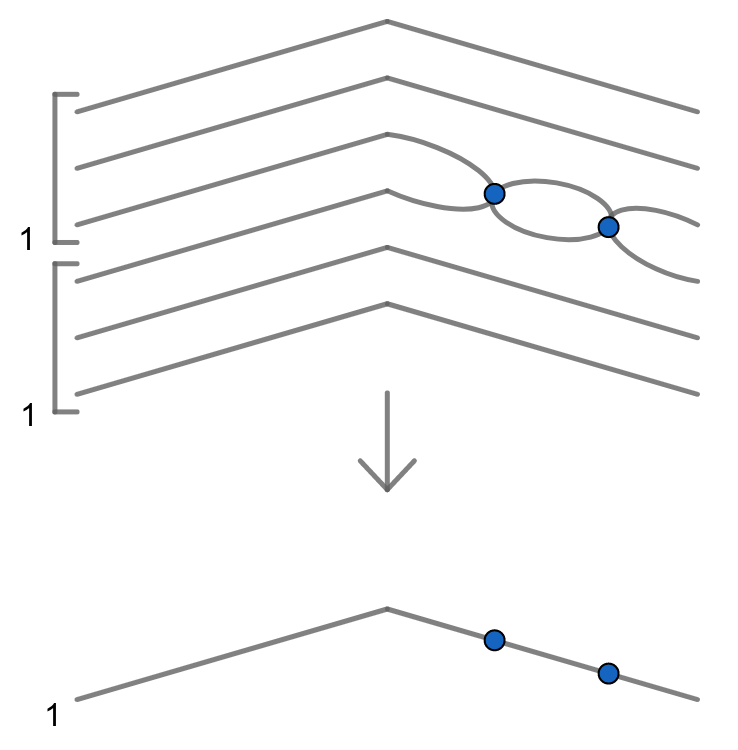}
     \caption{Cover of type $(\Delta_1,\Delta_1)$}\label{Fig:g2_Delta1-Delta1}
   \end{minipage}
\end{figure}

\subsubsection{$[Y]\in\Delta_0$}\label{adm_g2_delta0}

Let $[f:X\to Y]\in\Adm_{2/1,d}$ be an admissible cover, where $Y$ is an irreducible nodal curve of genus 1. We have a diagram
\begin{equation*}
\xymatrix{
\widetilde{X} \ar[r]^{\nu_X} \ar[d]^{\widetilde{f}} & X \ar[d]^{f} \\
\bP^1 \ar[r]^{\nu_Y} & Y
}
\end{equation*}
where the maps $\nu_X,\nu_Y$ are normalizations. Let $y_0\in Y$ denote the node, and let $y',y''\in\bP^1$ be its pre-images under $\nu_Y$. Let $y_1,y_2,\in Y$ be the branch points of $f$; by abuse of notation, we also let $y_1,y_2\in\bP^1$ denote their pre-images under $\nu_Y$. Then, $\widetilde{f}$ is simply branched over $y_1,y_2$, possibly branched over $y',y''$, and unramified everywhere else.

Let $\widetilde{X}_1,\ldots,\widetilde{X}_k$ be the components of $\widetilde{X}$, and let $d_i$ be the degree of $\widetilde{X}_i$ over $\bP^1$. Let $s_i$ be the total number of points of $\widetilde{X}_i$ lying over $y'$ and $y''$. Then, the number of points of $f^{-1}(y)$ is $t=\frac{1}{2}\sum s_i$. We have $p_a(\widetilde{X})\ge1-k$. On the other hand, $\widetilde{X}$ is the blowup of $X$ at $t$ nodes, so $p_a(\widetilde{X})=2-t$. We also have
\begin{equation*}
p_a(\widetilde{X})=1-h^0(\widetilde{X},\cO_{\widetilde{X}})+h^1(\widetilde{X},\cO_{\widetilde{X}})\ge 1-k,
\end{equation*}
hence $t\le k+1$. But note that $s_i\ge2$ for each $i$, so $t\ge k$.

Because $t$ is an integer, the three possibilities for the $s_i$ (up to re-indexing) are: $s_i=2$ for all $i$ (type $(\Delta_{0},\Delta_{0})$, Figure \ref{Fig:g2_Delta0-Delta0}), $s_1=4$ and $s_i=2$ for all $i\ge2$ (type $(\Delta_{00},\Delta_0)$, Figure \ref{Fig:g2_Delta00-Delta0}), and $s_1=s_2=3$ and $s_i=2$ for all $i\ge3$. In the last case, it is easy to check that $X^s$ will lie in one of the zero-dimensional boundary strata of $\barM_2$, so we may disregard covers of this type.

Suppose $f$ is a cover of type $(\Delta_{0},\Delta_{0})$. In this case, we will rename the number of components of $X$ by $m=k$. Then, $X$ must consist of a smooth genus 1 component $X_1$ attached at two points to a chain of $m-1$ rational curves, each of which maps to $Y$ via $x\mapsto x^a$. (Note that $a\ge2$, as $X_1$ has degree $a$ over $\bP^1$.) The map $f|_{X_1}:X_1\to\bP^1$ is totally ramified at the two nodes on $X_1$, and is simply ramified at two other points on $X_1$. We may also have $m=1$, in which case $X$ is irreducible with a single node, and its normalization $\wt{X}=X_1$ maps to $\bP^1$ as above.

Finally, suppose $f$ is a cover of type $(\Delta_{00},\Delta_{0})$. We have a single component $X_0\subset X$ with four points mapping to $y_0\in Y$; all other components of $X$ have two points mapping to $y_0$. As $X_0$ is connected, we see that $X$ must consist of two disjoint chains of curves $X_1,\ldots,X_{m-1}$ and $X'_1,\ldots,X'_{n-1}$ attached at two points to $X_0$. We allow one or both of $m,n$ to be equal to 1; in this case, $X_0$ has a non-separating node. We first assume $m,n>1$, in which case all components of $X$ have genus 0.

Each of the components of $X$ other than $X_0$ is unramified over $\bP^1$ away from the two nodes; thus, each $X_i\to\bP^1$ is of the form $x\mapsto x^a$ for some $a$ (independent of $i$), branched over $y',y''$. Similarly each $X'_j\to\bP^1$ is of the form $x\mapsto x^b$ for some $b$ (independent of $j$), branched over $y',y''$.

Now, $X_0$ has degree $a+b$ over $\bP^1$, and each of $y',y''$ has two points in its pre-image, of ramification indices $a$ and $b$. By Riemann-Hurwitz, there are two additional simple ramification points on $X_0$ mapping to $y_1,y_2\in Y$.

The situation is similar when at least one of $m,n$ is equal to 1: the chains of smooth rational curves attached to $X_0$ are replaced with a non-separating node on $X_0$, and the normalization of $X_0$ maps to $\bP^1$ as before.

\begin{figure}[!htb]
   \begin{minipage}{0.48\textwidth}
     \centering
     \includegraphics[width=.6\linewidth]{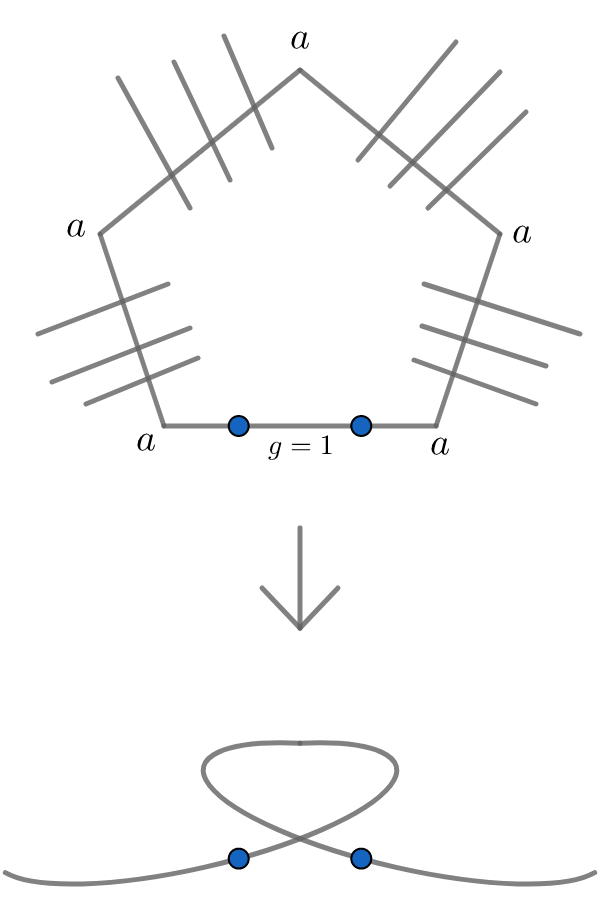}
     \caption{Cover of type $(\Delta_0,\Delta_0)$}\label{Fig:g2_Delta0-Delta0}
   \end{minipage}\hfill
   \begin{minipage}{0.48\textwidth}
     \centering
     \includegraphics[width=.9\linewidth]{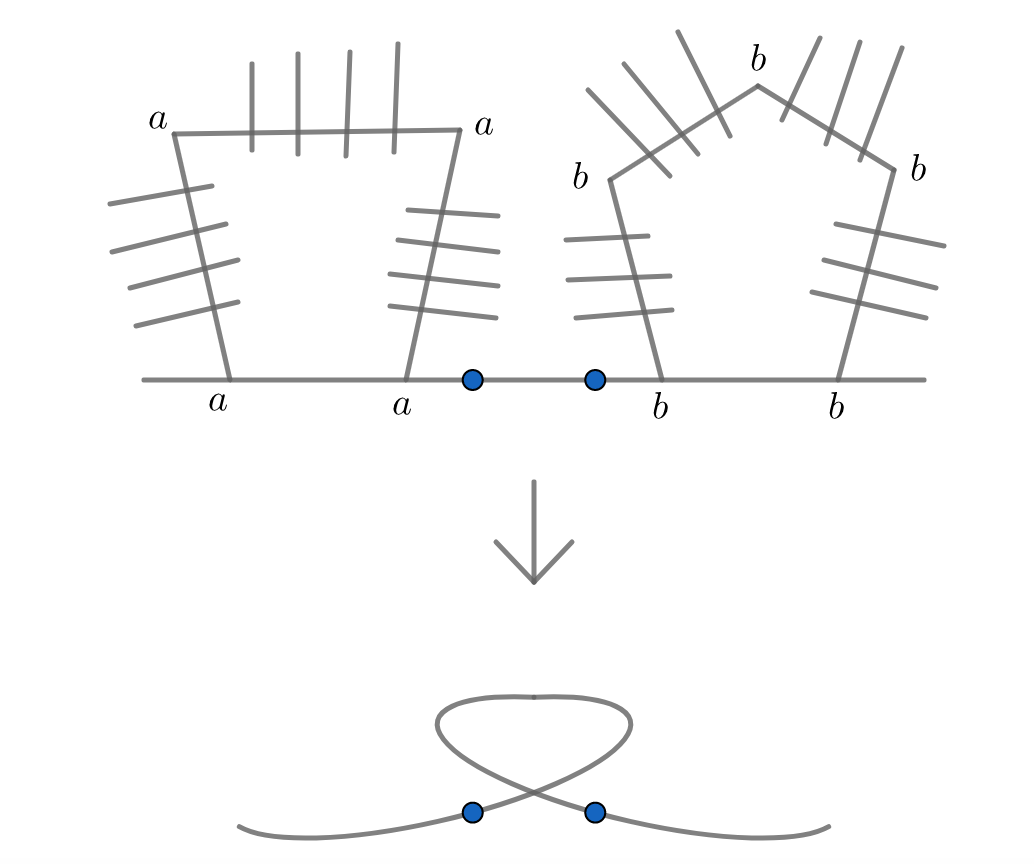}
     \caption{Cover of type $(\Delta_{00},\Delta_0)$}\label{Fig:g2_Delta00-Delta0}
   \end{minipage}
\end{figure}

\subsection{Intersection numbers: the case $[C]\in\Delta_{1}$}\label{adm_g2_intersect_delta1}

Suppose that we have a general curve class $C\to\barM_{1,1}\times\barM_{1,1}$. In the Cartesian diagram
\begin{equation*}
\xymatrix{
A_{C} \ar[rr] \ar[d] & & \Adm_{2/1,d} \ar[d]^{\pi_{2/1,d}}\\
C \ar[r] & \barM_{1,1}\times\barM_{1,1} \ar[r] & \barM_2
}
\end{equation*}
we wish to compute the degree of $A_C$ as a 0-cycle on $\barM_{2}$. 

Suppose $d_1,d_2$ are positive integers satisfying $d_1+d_2=d$. We also form the diagram

\begin{equation*}
\xymatrix{
A^{d_1,d_2}_{C} \ar[r] \ar[dd] & \Adm_{1/1,d_1}\times_{\Delta}\Adm_{1/1,d_2} \ar[rr] \ar[d] & &  \barM_{1,1}  \ar[d]^{\Delta} \\
 & \Adm_{1/1,d_1}\times\Adm_{1/1,d_2} \ar[rr]^(0.55){\psi_{1/1,d_1}\times\psi_{1/1,d_2}} \ar[d]^{\pi_{1/1,d_1}\times\pi_{1/1,d_2}} & &  \barM_{1,1}\times\barM_{1,1} \\\
C \ar[r] & \barM_{1,1}\times\barM_{1,1} & 
}
\end{equation*}
where both squares are Cartesian.

\begin{lem}
We have a bijection of sets
\begin{equation*}
A_C(\Spec(\bC))\cong\coprod_{d_1+d_2=d}A^{d_1,d_2}_C(\Spec(\bC))
\end{equation*}
In particular, the groupoids of geometric points of $A_C$ and $A^{d_1,d_2}_C$ are in fact sets, i.e., have no non-trivial automorphisms.
\end{lem}

\begin{proof}
By \S\ref{admissible_classification_g2}, a geometric point of $A_C$ consists of a point $x\in C$ and a cover $f:X\to Y$ of type $(\Delta_1,\Delta_1)$, along with the data of an isomorphism of $X^s$ with the curve corresponding to the image of $x$ in $\barM_2$. It is clear that this data has no non-trivial automorphisms.

From $f$, we may associate an \textit{ordered} pair of covers of the same elliptic curve, whose degrees are integers $d_1,d_2$ satisfying $d_1+d_2=d$. We thus get a geometric point of $A^{d_1,d_2}_C$, and it is again easy to see that such points have no non-trivial automorphisms.

The construction of the inverse map is clear: we need only note that $C$ may be chosen to avoid the point $([E_0],[E_0])\in\barM_{1,1}\times\barM_{1,1}$, where $E_0$ denotes a singular curve of genus 1; thus, all covers in $A^{d_1,d_2}_C$ are covers of (smooth) elliptic curves.
\end{proof}

\begin{lem}\label{g2_delta1,delta1_mult}
The intersection multiplicity at all points of $A^{d_1,d_2}_C$ is 1, and the intersection multiplicity at all points of $A_C$ is 2.
\end{lem}

\begin{proof}
Analytically locally near a point of $A^{d_1,d_2}_C$, the map $\Adm_{1/1,d_1}\times_{\Delta}\Adm_{1/1,d_2} \to\barM_{1,1}\times\barM_{1,1}$ is the inclusion of a smooth curve in a smooth surface, by Corollary \ref{adm_m11_m11} and the smoothness of $\barM_{1,1}$. Thus, $C\subset\barM_{1,1}\times\barM_{1,1}$ may be chosen to intersect $\Adm_{1/1,(d_1,d_2)}$ transversely.

Now, let $f:X\to Y$ be an admissible cover of type $(\Delta_1,\Delta_1)$. The complete local ring of $\Adm_{2/1,d}$ at $[f]$ is isomorphic to $\bC[[s,t]]$, where $t$ is a smoothing parameter for the node of $Y$, and the quotient $\bC[[s,t]]\to\bC[[s]]$ corresponds to the universal deformation of the elliptic component of $Y$. Let $\bC[[x,y,z]]$ be the complete local ring of $\barM_{2}$ at $[X^s]$, where $z$ is a smoothing parameter for the node of $X^s$, and the variables $x,y$ are deformation parameters for the two elliptic components.

Consider the induced map $T:\bC[[x,y,z]]\to\bC[[s,t]]$. We have the following, up to harmless renormalizations of the coordinates:
\begin{itemize}
\item $T(z)\equiv t^2\mod{(s^2,st,t^3)}$. To see this, consider any 1-parameter deformation of $Y$ that smooths the node to first order. The corresponding deformation of $f$ smooths the nodes of $X$ to first order, and the induced deformation of $X^s$ is obtained by contracting the rational bridge of $X$ in the total space, introducing an ordinary double point. It follows that the node of $X^s$ is smoothed to order 2.
\item $T(x)\equiv T(y)\equiv s\mod{(t,s^2)}$. Indeed, varying the elliptic component of $Y$ to first order varies the elliptic components of $X$ to first order.
\end{itemize}

The map on complete local rings at $[X^s]$ induced by $C\mapsto\barM_{1,1}\times\barM_{1,1}\to\barM_2$ is of the form $\bC[[x,y,z]]\mapsto\bC[[x,y]]\mapsto\bC[[u]]$, where $x,y$ map to power series with generic linear leading terms. It is straightforward to check that the complete local ring of $A_C$ at $([f],[X^s])$ is isomorphic to $\bC[t]/(t^2)$.
\end{proof}

\begin{prop}\label{m2_delta1,delta1}
The degree of $A_C$ as a 0-cycle on $\barM_2$ is:
\begin{equation*}
2\left(\sum_{d_1+d_2=d}\sigma_1(d_1)\sigma_1(d_2)\right)\int_{\barM_{1,1}\times\barM_{1,1}}[C]\cdot[\Delta]
\end{equation*}
where $[\Delta]=[p\times\barM_{1,1}]+[\barM_{1,1}\times p]$ is the class of the diagonal in $\barM_{1,1}\times\barM_{1,1}$, and $p$ is the class of a geometric point in $\barM_{1,1}$.
\end{prop}

\begin{proof}
By the previous two lemmas, it suffices to show that the degree of $A^{d_1,d_2}_C$ is:
\begin{equation*}
\sigma_1(d_1)\sigma_1(d_2)\int_{\barM_{1,1}\times\barM_{1,1}}[C]\cdot[\Delta]
\end{equation*}
Let $E$ be any elliptic curve. Writing $[\Delta]=[p\times\barM_{1,1}]+[\barM_{1,1}\times p]$, where we take $p$ to be the class of $[E]\in\barM_{1,1}$, the contribution of the first summand to $ \Adm_{1/1,d_1}\times_{\Delta}\Adm_{1/1,d_2}$ is $\Adm_{1/1,d_1}(E)\times\Adm_{1/1,d_2}$, where the first factor is the moduli space of degree $d_1$ covers of the fixed curve $E$. By \S\ref{m11_d-ell}, this class pushes forward to $\sigma_1(d_1)\sigma_1(d_2)[p\times\barM_{1,1}]$ on $\barM_{1,1}\times\barM_{1,1}$. Similarly, the second summand gives the class $\sigma_1(d_1)\sigma_1(d_2)[\barM_{1,1}\times p]$, so adding these contributions and intersecting with $C$ yields the result.
\end{proof}

\begin{rem}
An alternative approach to this computation is to work on the space $\barM_{2,2}$, where the ramification points of the $d$-elliptic cover are marked, and then push forward at the end to $\barM_{2}$. This approach has the advantage that unstable components with ramification points are not contracted in the map $\pi_{2/1,d}:\Adm_{2/1,d}\to\barM_{2,2}$, and one does not see resulting intersection multiplicities as in Lemma \ref{g2_delta1,delta1_mult}. Indeed, this approach is carried out in \cite{schmittvanzelm}, where an algorithm is given to intersect loci of Galois admissible covers with tautological classes, and we will develop this method further in forthcoming work.
\end{rem}

\subsection{Intersection numbers: the case $[C]\in\Delta_0$}\label{adm_g2_intersect_delta0}

Given $C\to\barM_{1,2}$ general, we wish to compute the degree of $A_C$, as defined below:

\begin{equation*}
\xymatrix{
A_{C} \ar[rr] \ar[d] & & \Adm_{2/1,d} \ar[d]^{\pi_{2/1,d}}\\
C \ar[r] & \barM_{1,2} \ar[r] & \barM_2
}
\end{equation*}

By \S\ref{admissible_classification_g2}, we have three topological types of covers contributing to $A_C$: type $(\Delta_0,\Delta_1)$, type $(\Delta_0,\Delta_0)$, and type $(\Delta_{00},\Delta_0)$. We now consider each contribution separately.

\subsubsection{Contribution from type $(\Delta_0,\Delta_1)$}

Consider the Cartesian diagram
\begin{equation*}
\xymatrix{
A^{(\Delta_{0},\Delta_1)}_{C} \ar[r] \ar[d] & \Adm_{1/1,d,2} \ar[d]^{\pi_{1/1,d,2}}\\
C \ar[r] & \barM_{1,2}
}
\end{equation*}
where $\Adm_{1/1,d,2}$ and its forgetful map to $\pi_{1/1,d,2}:\Adm_{1/1,d,2}\to\barM_{1,2}$ are as defined in \S\ref{m12_d-ell}.

\begin{prop}\label{m2_delta0,delta1}
The contribution to $A_C$ from covers of type $(\Delta_0,\Delta_1)$ is: 
\begin{equation*}
2\int_{\barM_{1,2}}[C]\cdot[\pi_{1/1,d,2}].
\end{equation*}
\end{prop}

\begin{proof}
It is easy to check that $A^{(\Delta_0,\Delta_1)}_C(\Spec(\bC))$ is a set, and is isomorphic to the subgroupoid of $A_C(\Spec(\bC))$ consisting of covers of type $(\Delta_0,\Delta_1)$. If $C$ is general, it intersects $\pi_{1/1,d,2}$ transversely (see, for example, the proof of Proposition \ref{m12_d-ell_class}) on $\barM_{1,2}$. On the other hand, an argument analogous to the proof of Lemma \ref{g2_delta1,delta1_mult} shows that a general $C$ intersects $\pi_{2/1,d}$ with multiplicity 2 at any cover $f:X\to Y$ of type $(\Delta_0,\Delta_1)$, due to the contraction of the rational bridge of $X$ after applying $\pi_{2/1,d}$.
\end{proof}

\subsubsection{Contribution from type $(\Delta_0,\Delta_0)$}

Fix positive integers $a,m$ satisfying $am=d$. Consider the Cartesian diagram
\begin{equation*}
\xymatrix{
A^{(\Delta_{0},\Delta_0),a}_{C} \ar[r] \ar[d] & \Adm_{1/0,a}^{a,a,2,2} \ar[d]^{\pi_{1/0,a}^{a,a,2,2}}\\
C \ar[r] & \barM_{1,2}
}
\end{equation*}
where $\Adm_{1/0,a}^{a,a,2,2}$ is as defined in \S\ref{dd22_pencils}, and the map $\pi_{1/0,a}^{a,a,2,2}:\Adm_{1/0,a}^{a,a,2,2}\to\barM_{1,2}$ remembers the (stabilized) source curve with the two total ramification points. Thus, the points of $A_C^{(\Delta_0,\Delta_0),a}$ record the main data of a cover of type $(\Delta_{0},\Delta_0)$, namely the restriction to the genus 1 component, which is a cover of $\bP^1$ totally ramified at two points. For a general $C$, all points of $A_C^{(\Delta_0,\Delta_0),a}$ correspond to covers of smooth curves.

\begin{prop}\label{m2_delta0,delta0}
The contribution to $A_C$ from covers of type $(\Delta_0,\Delta_0)$ is:
\begin{equation*}
\sum_{am=d}\left(m\int_{\barM_{1,2}}[C]\cdot[\pi_{1/0,a}^{a,a,2,2}]\right).
\end{equation*}
\end{prop}

\begin{proof}
It is easy to check that the geometric points of $A_C(\Spec(\bC))$ are in bijection with the geometric points of $A_C^{(\Delta_0,\Delta_0),a}$, where $a$ ranges over the positive integer factors of $d$. However, a cover $f:X\to Y$ of type $(\Delta_0,\Delta_0)$ has automorphism group of order $a^{m-1}$, as the group of $a$-th roots of unity acts on each rational component of $X$. On the other hand, each $A_C^{(\Delta_0,\Delta_0),a}(\Spec(\bC))$ is a set. If $C$ is general, it intersects $\pi_{1/0,a}^{a,a,2,2}$ transversely (see, for example, the proof of Proposition \ref{dd22_class_m13}).

It now suffices to show that the intersection multiplicity (on the level of complete local rings) of $C$ with $\pi_{2/1,d}$ at a cover $f:X\to Y$ of type $(\Delta_0,\Delta_0)$ is $ma^{m-1}$; after dividing by the order of the automorphism group, we get the factor of $m$ in the statement, and the conclusion follows.

By Proposition \ref{adm_defos}, the complete local ring at $[f]$ is
\begin{equation*}
\bC[[s,t,t_1,\ldots,t_m]]/(t=t_1^a=\cdots=t_m^a),
\end{equation*}
which is canonically a $\bC[[s,t]]$-algebra via $\psi_{2/1,d}:\Adm_{2/1,d}\to\barM_{1,2}$. Here, $t$ is a smoothing parameter for the node of $Y$ and $t_i$ is a smoothing parameter for the node $x_i\in X$. The quotient $\bC[[s]]$ corresponds to the deformation of the target that moves the marked points $y_1,y_2\in Y$ apart.

Let $\bC[[x,y,z]]$ be the complete local ring of $\barM_2$ at $[X^s]$, where the coordinates are chosen as follows. The coordinate $z$ is the smoothing parameter for the node, so that $\bC[[x,y]]$ is the deformation space of the marked normalization $(X_1,x_1,x_m)$ of $X^s$ (the points $x_1,x_m$ are the nodes along $X_1\subset X$). Then, $y$ is the deformation parameter of the elliptic curve $(X_1,x_1)$, and the quotient $\bC[[x]]$ corresponds to the deformation of $X$ moving $x_1$ and $x_m$ apart.

Consider the induced map on complete local rings
\begin{equation*}
T:\bC[[x,y,z]]\to \bC[[s,t,t_1,\ldots,t_m]]/(t=t_1^a=\cdots=t_m^a).
\end{equation*}

We have the following, up to harmless renormalizations of the coordinates:
\begin{itemize}
\item $T(z)\equiv t_1^m\mod{(s,t_1-t_2,\ldots,t_1-t_n)}$. Indeed, in the 1-parameter deformation of $f$ smoothing the nodes of $X$ to first order, consider the total space of the associated deformation of $X$. Contracting the rational components of $X$ produces an $A_m$-singularity, so the node of $X^s$ is smoothed to order $m$ in its induced deformation.
\item $T(z)\equiv 0\mod{(t_i)}$, as a deformation that keeps the node $x_i$ in $X$ also keeps the node in $X^s$.
\item $T(y)\equiv s\mod{(t_1,\ldots,t_n)}$. The content here is that first-order deformation of $[f]$ that moves $y_1,y_2$ apart varies the elliptic curve $(X_1,x_1)$ to first order. Indeed, the map $\pi_{1/0,a}^{a,a,2,2}$ is unramified over a general point of $M_{0,4}$, and $\psi_{1/0,a}^{a,a,2,2}$ is unramified over $\cM_{1,1}$, so we have the claim for $C$ general.
\end{itemize}

The first two claims imply that $T(z)=(t_1\cdots t_m)u$, where $u$ is a unit. Then, it is straightforward to check that the complete local ring of $A_C$ at $([f],[X^s])$ is
\begin{equation*}
\bC[t,t_1,\ldots,t_m]/(t-t_i^a,t_1\cdots t_m),
\end{equation*}
which has dimension $ma^{m-1}$ as a $\bC$-vector space (for example, a basis is given by monomials $t^et_1^{e_1}\cdots t_m^{e_m}$, where $0\le e\le m-1$ and $0\le e_i\le a-1$). This completes the proof.
\end{proof}

\subsubsection{Contribution from type $(\Delta_{00},\Delta_0)$}

\begin{lem}\label{aabb_maps_from_P1}
Let $x_1,x_2,x_3,x_4\in\bP^1$ be four distinct points, and let $a,b\ge1$ be integers. Then, up to scaling on the target, there is a unique cover $g:\bP^1\to\bP^1$ of degree $a+b$ such that $f$ has zeroes of orders $a,b$ at $x_1,x_2$, respectively, and poles of orders $a,b$ at $x_3,x_4$, respectively. If the $x_i$ are general, then $g$ is simply ramified over two other distinct points.
\end{lem}

\begin{proof}
The unique such map is the meromorphic function
\begin{equation*}
g(x)=\frac{(x-x_1)^a(x-x_2)^b}{(x-x_3)^a(x-x_4)^b}.
\end{equation*}
The second half of the statement follows from Riemann-Hurwitz and a dimension count, as there are finitely many covers of $\bP^1$ branched over 3 points.
\end{proof}

Consider the Cartesian diagram
\begin{equation*}
\xymatrix{
A^{(\Delta_{00},\Delta_0)}_{C} \ar[r] \ar[d] & \overline{M}_{0,4} \ar[d]\\
C \ar[r] & \barM_{1,2}
}
\end{equation*}
where the map $\overline{M}_{0,4}\to\barM_{1,2}$ glues the third and fourth marked points; its class is $\Delta_0\in A^1(\barM_{1,2})$. For generic $C$, the points of $A^{(\Delta_{00},\Delta_0)}_{C}$ correspond to the points on $C$ whose images in $\barM_{1,2}$ are irreducible nodal curves.

\begin{prop}\label{m2_delta00,delta0}
The contribution to $A_C$ from covers of type $(\Delta_{00},\Delta_0)$ is: 
\begin{equation*}
2\left(\sum_{d_1+d_2=d}\sigma_1(d_1)\sigma_1(d_2)\right)\int_{\barM_{1,2}}[C]\cdot\Delta_0
\end{equation*}
\end{prop}

\begin{proof}
It is clear that $A^{(\Delta_{00},\Delta_0)}_{C}(\Spec(\bC))$ is a set. Given positive integers $a,b,m,n$ satisfying $am+bn=d$, a geometric point of $A^{(\Delta_{00},\Delta_0)}_{C}$ gives rise to a unique cover $g$ as in Lemma \ref{aabb_maps_from_P1}, which, for general $C$, will be simply ramified over two points distinct from each other and from $0,\infty$. Then, $g$ gives rise to two admissible covers of type $(\Delta_{00},\Delta_0)$, distinguished by the labelling of the two simple ramification points; all such covers are obtained (uniquely) in this way.

Each cover $f:X\to Y$ of type $(\Delta_{00},\Delta_0)$ has automorphism group of order $a^{m-1}b^{n-1}$, coming from the actions of roots of unity on the rational components of $X$. Now, consider the complete local rings of $A_C$ at $([f],X^s)$. 

As in the proof of Proposition \ref{m2_delta0,delta0}, consider the map on complete local rings
\begin{equation*}
T:\bC[[x,y,z]]\to\bC[[s,t,t_1,\ldots,t_m,u_1,\ldots,u_n]]/(t-t_i^a,t-u_{j}^{b}),
\end{equation*}
induced by $\pi_{2/1,d}$. We take $y$ to be the smoothing parameter of the node obtained from the gluing map $\barM_{1,2}\to\barM_2$, $z$ to be that coming from the map $\overline{M}_{0,4}\to\barM_{1,2}$, and $x$ such that the quotient $\bC[[x]]$ corresponds to the deformation of $X^s$ along $\overline{M}_{0,4}$. We verify that, up to renormalizing coordinates,
\begin{itemize}
\item $T(x)=s\mod(t_1,\ldots,t_m,u_1,\ldots,u_n)$,
\item $T(y)=(t_1\ldots t_m)v$, and
\item $T(z)=(u_1\ldots u_n)v'$,
\end{itemize}
where $v,v'$ are units. The complete local ring of $A_{C}$ at $[f]$ is thus
\begin{equation*}
\bC[t,t_1,\ldots,t_m,u_1,\ldots,u_n]/(t-t_i^a,t-u_j^b,u_1\cdots u_n),
\end{equation*}
which has a $\bC$-basis given by monomials $t^e(t_1^{e_1}\cdots t_{m-1}^{e_{m-1}})(u_1^{f_1}\cdots u_n^{f_n})$, where $0\le e\le m-1$, $0\le e_i\le a-1$, and $0\le f_j\le b-1$. The total contribution to $A_C$ of each $f$ is therefore $(ma^{m-1}b^{n})/(a^{m-1}b^{n-1})=mb$.

Summing over all $a,m,b,n$, we obtain that each point of intersection of $C$ and $\Delta_0$ contributes
\begin{equation*}
2\left(\sum_{am+bn=d}mb\right)=2\left(\sum_{d_1+d_2=d}\sigma_1(d_1)\sigma_1(d_2)\right)
\end{equation*}
to $A_C$.
\end{proof}

\subsection{The class of the admissible locus}

We are now ready to compute the class of the $d$-elliptic locus in genus 2, that of $\pi_{2/1,d}:\Adm_{2/1,d}\to\barM_2$.

\begin{prop}\label{adm_m2_intersections}
We have:
\begin{align*}
\int_{\barM_{2}}[\pi_{2/1,d}]\cdot\Delta_{00}&=4(d-1)\sigma_1(d)\\
\int_{\barM_{2}}[\pi_{2/1,d}]\cdot\Delta_{01}&=2\left(\sum_{d_1+d_2=d}\sigma_1(d_1)\sigma_1(d_2)\right)
\end{align*}
\end{prop}

\begin{proof}
Note first that the formulas of Propositions \ref{m2_delta1,delta1}, \ref{m2_delta0,delta1}, \ref{m2_delta0,delta0}, and \ref{m2_delta00,delta0} hold for \textit{any} curve class defined on the relevant boundary divisor, as such a class may be written as a linear combination of general curves, and the formulas are all linear in $[C]$.

We first take $\Delta_{01}$ as the pushforward of $[p\times\barM_{1,1}]\in A_1(\barM_{1,1}\times\barM_{1,1})$ to $\barM_{2}$. We have $\int_{\barM_{1,1}\times\barM_{1,1}}\Delta_{01}\cdot[\Delta]=1$, so the formula for $\int_{\barM_{2}}[\pi_{2/1,d}]\cdot\Delta_{01}$ follows immediately from Proposition \ref{m2_delta1,delta1}.

As a check, $\Delta_{01}$ may also be expressed as the pushforward of $\Delta_1\in A_1(\barM_{1,2})$ to $\barM_2$, so we can apply Propositions \ref{m2_delta0,delta1}, \ref{m2_delta0,delta0}, and \ref{m2_delta00,delta0}. Combining these with Proposition \ref{m12_d-ell_class}, Proposition \ref{dd22_class_m13}, and \S\ref{int_numbers_m12}, respectively, we find that the first two contributions are zero, and the third is
\begin{equation*}
2\left(\sum_{d_1+d_2=d}\sigma_1(d_1)\sigma_1(d_2)\right),
\end{equation*}
so we obtain the same result.

Finally, $\Delta_{00}$ is the pushforward of $\Delta_0\in A_1(\barM_{1,2})$ to $\barM_2$. Applying the same formulas as above, we get a contribution of $2(d-1)\sigma_1(d)$ in type $(\Delta_0,\Delta_1)$, a contribution of
\begin{equation*}
\sum_{am=d}2(a^2-1)m=2(d-1)\sigma_1(d)
\end{equation*}
in type $(\Delta_0,\Delta_0)$ (where we have applied the projection formula for the forgetful morphism $\barM_{1,3}\to\barM_{1,2}$, under which $\Delta_0$ pulls back to $\Delta_0$), and zero in type $(\Delta_{00},\Delta_0)$.
\end{proof}

\begin{proof}[Proof of Theorem \ref{main_thm_g2}]
Immediate from of \S\ref{int_numbers_m2} and the fact that $\delta_i=\frac{1}{2}\Delta_i$ for $i=0,1$, along with the identity
\begin{equation*}
\sum_{d_1+d_2=d}\sigma_1(d_1)\sigma_1(d_2)=\left(-\frac{1}{2}d+\frac{1}{12}\right)\sigma_1(d)+\frac{5}{12}\sigma_3(d),
\end{equation*}
see \S\ref{qmod_prelim}.
\end{proof}


\section{Variants in genus 2}\label{g2_var}

\subsection{Covers of a fixed elliptic curve}\label{m2_d-ell_fixed_target}

Fix a general elliptic curve $E$; we consider genus 2 curves covering $E$. Define the space of such covers $\Adm_{2/1,d}(E)$ by the Cartesian diagram
\begin{equation*}
\xymatrix{
\Adm_{2/1,d}(E) \ar[r] \ar[d] & \Adm_{2/1,d} \ar[d] \\
[E] \ar[r] & \barM_{1,1},
}
\end{equation*}
where the map $\Adm_{2/1,d}\to\barM_{1,1}$ is the composition of $\psi_{2/1,d}:\Adm_{2/1,d}\to\barM_{1,2}$ with the map forgetting the second marked point.

By post-composing with $\pi_{2/1,d}$, we get a map $\pi_{2/1,d}(E):\Adm_{2/1,d}(E)\to\barM_{2}$; we wish to compute its class in $A^2(\barM_2)$. We do so by intersecting with the boundary divisors $\Delta_0$ and $\Delta_1$. If $f:X\to Y$ is an admissible cover appearing in one of these intersections, then $Y=E\cup\bP^1$, where both components are attached at a node and both marked points are on the rational component, and $f$ is of type $(\Delta_0,\Delta_1)$ or $(\Delta_1,\Delta_1)$.

\begin{prop}[{cf. \cite[Proposition 2.1]{chen}}]\label{adm_m2E_intersections}
We have
\begin{align*}
\int_{\barM_2}\pi_{2/1,d}(E)\cdot\Delta_0&=2(d-1)\sigma_1(d)\\
\int_{\barM_2}\pi_{2/1,d}(E)\cdot\Delta_1&=2\left(\sum_{d_1+d_2=d}\sigma_1(d_1)\sigma_1(d_2)\right)
\end{align*}
\end{prop}

\begin{proof}
The points of intersection of $\pi_{2/1,d}(E)$ and the gluing morphism $\barM_{1,2}\to\barM_2$ consist of admissible covers of type $(\Delta_0,\Delta_1)$ with target $Y$, which correspond to isogenies $E'\to E$ with a second marked point in the kernel. As in Proposition \ref{m2_delta0,delta1}, each cover appears with multiplicity 2, so the first claim follows from Lemma \ref{count_pointed_isogenies}.

The points of intersection of $\pi_{2/1,d}(E)$ and the gluing morphism $\barM_{1,1}\times\barM_{1,1}\to\barM_2$ consist of admissible covers of type $(\Delta_1,\Delta_1)$ with target $Y$, which correspond to ordered pairs of isogenies $E_1\to E,E_2\to E$. As in Lemma \ref{g2_delta1,delta1_mult}, each such admissible cover appears with multiplicity 2, and it is easy to check that the geometric points have no non-trivial automorphisms. We are now done by Lemma \ref{m11_d-ell}.
\end{proof}

Using the fact that $\delta_{00}=\frac{1}{8}\Delta_{00}$ and $\delta_{01}=\frac{1}{2}\Delta_{01}$, and applying again the identity 
\begin{equation*}
\sum_{d_1+d_2=d}\sigma_1(d_1)\sigma_1(d_2)=\left(-\frac{1}{2}d+\frac{1}{12}\right)\sigma_1(d)+\frac{5}{12}\sigma_3(d)
\end{equation*}
we conclude:
\begin{thm}\label{adm_m2E_formula}
The class of $\pi_{2/1,d}(E)$ in $A^2(\barM_2)$ is:
\begin{equation*}
\left(-4d\sigma_1(d)+4\sigma_3(d)\right)\delta_{00}+
\left(-4\sigma_1(d)+4\sigma_3(d)\right)\delta_{01}.
\end{equation*}
\end{thm}

\subsection{Interlude: quasimodularity for correspondences}\label{m2_correspondence}

\begin{thm}\label{correspondence_statement}
Consider the correspondence
\begin{equation*}
(\pi_{2/1,d})_{*}\circ\psi_{2/1,d}^{*}:A^{*}(\barM_{1,2})\to A^{*}(\barM_2).
\end{equation*}
Then, for a fixed $\alpha\in A^{*}(\barM_{1,2})$, we have
\begin{equation*}
(\pi_{2/1,d})_{*}\circ\psi_{2/1,d}^{*}(\alpha)\in A^{*}(\barM_2)\otimes\Qmod.
\end{equation*}
\end{thm}

\begin{proof}
It suffices to check the claim on a basis of $A^{*}(\barM_{1,2})$. Note that all classes of geometric points on $\barM_{1,2}$ or $\barM_{2}$ are rationally equivalent, as both spaces are unirational. When $\alpha$ is the class of a point, we have Theorem \ref{dijkgraaf}. When $\alpha$ is the fundamental class, the claim follows from Theorem \ref{modularity_statement} (in genus 2). When $\alpha=\Delta_0$, we may replace the locus of covers of a nodal curve of genus 1 with that of covers of a fixed smooth curve, in which case we are done by Theorem \ref{adm_m2E_formula}.

It remains to consider $\alpha=\Delta_1$. Consider the intersection of a general divisor $D\to\barM_2$ with the cycle $(\pi_{2/1,d})_{*}\circ\psi_{2/1,d}^{*}\Delta_1$. The contribution from covers of type $(\Delta_0,\Delta_1)$ to the intersection of $D$ with the admissible locus is the intersection of $D$ with the pushforward of $\pi_{1/1,d,2}$ to $\barM_2$, which is quasimodular by Proposition \ref{m12_d-ell_class}. The contribution from type $(\Delta_1,\Delta_1)$ is the intersection of $D$ with $\Adm_{1/1,d_1}\times_{\Delta}\Adm_{1/1,d_2}$, where $d_1,d_2$ range over integers satisfying $d_1+d_2=d$; this is also quasimodular by Corollary \ref{adm_m11_m11}.
\end{proof}

\subsection{The $d$-elliptic locus on $\barM_{2,1}$}\label{m21_d-ell}

Here, we compute the class in $A^2(\barM_{2,1})$ of the morphism $\pi_{2/1,d}:\Adm_{2/1,d}\to\barM_{2,1}$, whose image is the closure of the locus of pointed curves $(C,p)$ admitting a degree $d$ cover of an elliptic curve, ramified at $p$. We do so by intersecting with test surfaces, following the same method as in \S\ref{m2_d-ell}. Because $A_2(\cM_{2,1})=0$ (see \S\ref{int_numbers_m21}), it suffices to consider test surfaces in the boundary of $\barM_{2,1}$, which is the union of the boundary divisors
\begin{enumerate}
\item[$(\Delta_0)$] $\barM_{1,3}\to\barM_{2,1}$
\item[$(\Delta_1)$] $\barM_{1,2}\times\barM_{1,1}\to\barM_{2,1}$
\end{enumerate}
Here, the map $\barM_{1,3}\to\barM_{2,1}$ glues together the second and third marked points, and the map $\barM_{1,2}\times\barM_{1,1}\to\barM_{2,1}$ glues the second marked point on the first component to the marked point of the second. As in the unpointed case, a general surface $S$ mapping to one of these boundary classes will intersect the admissible locus at covers of one of the four types described in \S\ref{admissible_classification_g2}.

\subsubsection{The case $[S]\in\Delta_1$}

All admissible covers $f:X\to Y$ in the intersection of $S\to\barM_{1,2}\times\barM_{1,1}\to\barM_{2,1}$ and $\pi_{2/1,d}:Adm_{2/1,d}\to\barM_{2,1}$ have type $(\Delta_1,\Delta_1)$. Let $s:\barM_{1,1}\to\barM_{1,2}$ be the map attaching a 2-pointed rational curve to an elliptic curve at its origin. Then, the 1-pointed curve $X$ is obtained by gluing a point of $\cM_{1,2}$ in the image of $s$, at its first marked point, to a point of $\cM_{1,1}$, at its origin.

For integers $d_1,d_2$ satisfying $d_1+d_2=d$, we have a diagram

\begin{equation*}
\xymatrix{
A^{d_1,d_2}_{S} \ar[r] \ar[dd] & \Adm_{1/1,d_1}\times_{\Delta}\Adm_{1/1,d_2} \ar[rr] \ar[d] & &  \barM_{1,1}  \ar[d]^{\Delta} \\
 & \Adm_{1/1,d_1}\times\Adm_{1/1,d_2} \ar[rr]^(0.57){\pi_{1/1,d_1}\times\pi_{1/1,d_2}}  & &  \barM_{1,1}\times\barM_{1,1} \ar[d]^{s\times\id}\\\
S \ar[rrr] & & & \barM_{1,2}\times\barM_{1,1} 
}
\end{equation*}
where both squares are Cartesian. Taking the union over all possible $(d_1,d_2)$, the geometric points of $A^{d_1,d_2}_{S}$ are in bijection with those of the intersection of $\Adm_{2/1,d}$ and $S$, and there are no non-trivial automorphisms on either side. In the map $\pi_{2/1,d}:\Adm_{2/1,d}\to\barM_{2,1}$, the rational bridge of $X$ in a cover $f:X\to Y$ of type $(\Delta_1,\Delta_1)$ does not get contracted, so in fact the argument of Lemma \ref{g2_delta1,delta1_mult} shows that both intersections are transverse for general $S$.

Using the fact that $\Delta=[p\times\barM_{1,1}]+[\barM_{1,1}\times p]$ and applying Corollary \ref{adm_m11_m11}, we find, as in Proposition \ref{m2_delta1,delta1}:

\begin{prop}\label{m21_delta1,delta1}
For $S\to\barM_{1,2}\times\barM_{1,1}$, we have:
\begin{equation*}
\int_{\barM_{2,1}}[S]\cdot[\pi_{2/1,d}]=\left(\int_{\barM_{1,2}\times\barM_{1,1}}([p\times\barM_{1,1}]+[\Delta_1\times p])\cdot[S]\right)\cdot\sum_{d_1+d_2=d}\sigma_1(d_1)\sigma_1(d_2).
\end{equation*}
\end{prop}

\subsubsection{The case $[S]\in\Delta_0$.}

In the intersection
\begin{equation*}
\xymatrix{
A_S \ar[r] \ar[d] & \Adm_{2/1,d} \ar[d]^{\pi_{2/1,d}} \\
S \ar[r] & \barM_{2,1}
}
\end{equation*}
we have contributions from covers of types $(\Delta_0,\Delta_1)$, $(\Delta_0,\Delta_0)$, and $(\Delta_{00},\Delta_0)$.

First, consider a cover $f:X\to Y$ of type $(\Delta_0,\Delta_1)$. The 1-pointed curve $X$ may be obtained by identifying the points $x_2,x_3$ of $[(X_1,x_1,x_2,x_3)]\in\barM_{1,3}$ lying in either of the boundary divisors $\Delta_{1,\{1,2\}},\Delta_{1,\{1,3\}}$. Let $s_{\{1,2\}},s_{\{1,3\}}:\barM_{1,2}\to\barM_{1,3}$ be the maps defining these two boundary divisors, we have the Cartesian diagrams
\begin{equation*}
\xymatrix{
A_S^{(\Delta_1,\Delta_0),\{1,2\}} \ar[r] \ar[dd] & \Adm_{1/1,d,2} \ar[d]^{\pi_{1/1,d,2}} \\
 & \barM_{1,2} \ar[d]^{s_{\{1,2\}}} \\
S\ar[r] & \barM_{1,3}
}
\hspace{0.5in}
\xymatrix{
A_S^{(\Delta_1,\Delta_0),\{1,3\}} \ar[r] \ar[dd] & \Adm_{1/1,d,2} \ar[d]^{\pi_{1/1,d,2}} \\
 & \barM_{1,2} \ar[d]^{s_{\{1,3\}}} \\
S\ar[r] & \barM_{1,3}
}
\end{equation*}
Following the proof of Proposition \ref{m2_delta0,delta1}, both intersections above are seen to be transverse (note that we no longer contract the rational bridge of $X$) and have no non-trivial automorphisms. We conclude:

\begin{prop}\label{m21_delta0,delta1}
The contribution to $A_S$ from covers of type $(\Delta_0,\Delta_1)$ is:
\begin{equation*}
\int_{\barM_{1,3}}[S]\cdot\left([s_{\{1,2\}}\circ\pi_{1/1,d,2}]+[s_{\{1,3\}}\circ\pi_{1/1,d,2}]\right)
\end{equation*}
\end{prop}

The analysis of covers of type $(\Delta_0,\Delta_0)$ is essentially identical to that of Proposition \ref{m2_delta0,delta0}. We find:

\begin{prop}\label{m21_delta0,delta0}
The contribution to $A_S$ from covers of type $(\Delta_0,\Delta_0)$ is:
\begin{equation*}
\sum_{am=d}\left(m\int_{\barM_{1,3}}[S]\cdot[\pi_{1/0,a}^{a,a,2,2}]\right).
\end{equation*}
\end{prop}

Finally, consider covers of type $(\Delta_{00},\Delta_0)$. Fix $a,b,m,n$ with $am+bn=d$, and let $\Adm_{0/0,a+b}^{(a,b),(a,b),2,2}$ be the space of tuples $(f,x_1,\ldots,x_6)$, where $f:X\to Y$ is a degree $a+b$ admissible cover of genus 0 curves, with six marked points $x_1,\ldots,x_6\in X$ such that $f(x_1)=f(x_2)$, $f(x_3)=f(x_4)$, and the ramification indices at $x_1,\ldots,x_6$ are $a,b,a,b,2,2$, respectively. As usual, there is a canonical morphism $\pi_{0/0,a+b}^{(a,b),(a,b),2,2}:\Adm_{0/0,a+b}^{(a,b),(a,b),2,2}\to\overline{M}_{0,6}$ remembering the pointed source curve. Let $r:\overline{M}_{0,6}\to\barM_{1,3}$ be the map sending
\begin{equation*}
[(X,x_1,\ldots,x_6)]\mapsto [(X/(x_2\sim x_4),x_5,x_1,x_3)^s]
\end{equation*}
We have a Cartesian diagram

\begin{equation*}
\xymatrix{
A_S^{(\Delta_{00},\Delta_0)} \ar[r] \ar[dd] & \Adm_{0/0,a+b}^{(a,b),(a,b),2,2} \ar[d]^{\pi_{0/0,a+b}^{(a,b),(a,b),2,2}} \\
 & \overline{M}_{0,6} \ar[d]^r \\
S \ar[r] & \barM_{1,3}
 }
 \end{equation*}
 
It is routine to check, following the proof of Proposition \ref{m2_delta00,delta0}:

\begin{prop}\label{m21_delta00,delta0}
The contribution to $A_S$ from covers of type $(\Delta_{00},\Delta_0)$ is 
\begin{equation*}
\sum_{am+bn=d}\left(mb\int_{\barM_{1,3}}[S]\cdot[r\circ\pi_{0/0,a+b}^{(a,b),(a,b),2,2}]\right).
\end{equation*}
\end{prop}

In order to complete the calculation, we will need to compute the class $[r\circ \pi_{0/0,a+b}^{(a,b),(a,b),2,2}]\in A_1(\barM_{1,3})$.

\begin{prop}\label{double_(a,b)_class}
We have:
\begin{align*}
\int_{\barM_{1,3}}[r\circ \pi_{0/0,a+b}^{(a,b),(a,b),2,2}]\cdot\Delta_0&=0\\
\int_{\barM_{1,3}}[r\circ \pi_{0/0,a+b}^{(a,b),(a,b),2,2}]\cdot\Delta_{1,T}&=1\text{ for }T=\{2,3\},\{1,2,3\}\\
\int_{\barM_{1,3}}[r\circ \pi_{0/0,a+b}^{(a,b),(a,b),2,2}]\cdot\Delta_{1,T}&=0\text{ for }T=\{1,2\},\{1,3\}
\end{align*}
(Recall from the definition of $r:\overline{M}_{0,6}\to\barM_{1,3}$ that the three marked points of a curve in $\barM_{1,3}$ in the image of $\overline{M}_{0,6}$ come from the points $x_5,x_1,x_3$ of the marked genus 0 curve.)
\end{prop}

\begin{proof}
For the first statement, we may replace $\Delta_0$ with the locus of pointed curves with a fixed underlying elliptic curve $(E,x_1)$, which clearly has empty intersection with $[r\circ \pi_{0/0,a+b}^{(a,b),(a,b),2,2}]$.

Now, consider a cover in $\Adm_{0/0,a+b}^{(a,b),(a,b),2,2}$ whose image in $\barM_{1,3}$ has a non-separating node. We claim that the only such cover, up to isomorphism, is constructed as follows. Let $Y$ be the union of two copies $Y_1,Y_2$ of $\bP^1$, attached at a node, with two marked points on each component. Then, $X$ contains two copies of $\bP^1$ mapping to $Y_1$ via the maps $x\mapsto x^a$ and $x\mapsto x^b$, respectively. These two components are connected by a copy of $\bP^1$ mapping to $Y_2$ via $x\mapsto x^2$, and the rest of the components of $X$ map to $Y_2$ isomorphically.

The cover $f:X\to Y$ constructed above gives a single point of intersection of the admissible locus with $\Delta_{1,\{2,3\}}$ and $\Delta_{1,\{1,2,3\}}$; it is now standard to check that the multiplicity is 1.
\end{proof}

\subsubsection{Final computation} We are now ready to intersect $[\pi_{2/1,d}]\in A_2(\barM_{2,1})$ with the boundary test surfaces.

\begin{prop}\label{adm_m21_intersections}
We have:
\begin{align*}
\int_{\barM_{2,1}}[\pi_{2/1,d}]\cdot\Delta_{00}&=4(d-1)\sigma_1(d)\\
\int_{\barM_{2,1}}[\pi_{2/1,d}]\cdot\Delta_{01a}&=\sum_{d_1+d_2=d}\sigma_1(d_1)\sigma_1(d_2)\\
\int_{\barM_{2,1}}[\pi_{2/1,d}]\cdot\Delta_{01b}&=\sum_{d_1+d_2=d}\sigma_1(d_1)\sigma_1(d_2)\\
\int_{\barM_{2,1}}[\pi_{2/1,d}]\cdot\Xi_1&=-\frac{1}{24}(d-1)\sigma_1(d)\\
\int_{\barM_{2,1}}[\pi_{2/1,d}]\cdot\Delta_{11}&=-\frac{1}{24}\left(\sum_{d_1+d_2=d}\sigma_1(d_1)\sigma_1(d_2)\right)
\end{align*}
\end{prop}

\begin{proof}
First, we consider the classes $\Delta_{01a},\Delta_{01b}$, and $\Delta_{11}$ contained in $\barM_{1,2}\times\barM_{1,1}$; these are the push-forwards of the boundary divisors $\barM_{1,2}\times p$, $\Delta_0\times\barM_{1,1}$, and $\Delta_1\times\barM_{1,1}$, respectively.
Thus,
\begin{align*}
\int_{\barM_{1,2}\times\barM_{1,1}}([p\times\barM_{1,1}]+[\Delta_1\times p])\cdot[\Delta_{01a}]&=1\\
\int_{\barM_{1,2}\times\barM_{1,1}}([p\times\barM_{1,1}]+[\Delta_1\times p])\cdot[\Delta_{01b}]&=1\\
\int_{\barM_{1,2}\times\barM_{1,1}}([p\times\barM_{1,1}]+[\Delta_1\times p])\cdot[\Delta_{11}]&=-\frac{1}{24},
\end{align*}
applying \S\ref{int_numbers_m12}. The formulas for the intersections of $\pi_{2/1,d}$ with these three classes now follow from Proposition \ref{m21_delta1,delta1}.

Now, we consider the classes $\Delta_{00},\Delta_{01a},\Delta_{01b},\Xi_1$ contained in $\barM_{1,3}$; these are the push-forwards of the boundary divisors $\Delta_0$, $\Delta_{1,\{2,3\}}$, $\Delta_{1,\{1,2,3\}}$, and $\Delta_{1,\{1,3\}}$, respectively. (We have included the middle two classes, which arose earlier, as a check.) 

By Propositions \ref{dd22_class_m13} and \ref{m21_delta0,delta0}, the only class for which covers of type $(\Delta_0,\Delta_0)$ contribute is $\Delta_{00}$, in which we get a contribution of $2(d-1)\sigma_1(d)$. By Propositions \ref{m21_delta00,delta0} and \ref{double_(a,b)_class}, covers of type $(\Delta_{00},\Delta_{0})$ contribute $\sum_{d_1+d_2=d}\sigma_1(d_1)\sigma_1(d_2)$ to each of $\Delta_{01a},\Delta_{01b}$ and nothing to the others.

Finally, consider covers of type $(\Delta_0,\Delta_1)$. By Corollary \ref{m12_d-ell_class_formula}, we have
\begin{align*}
[s_{\{2,3\}}\circ\pi_{1/1,d,2}]&=(d-1)\sigma_1(d)\left(\frac{1}{24}\Delta_{01,\{1,2\}}+\Delta_{11,\{1,2\}}\right)\\
[s_{\{1,3\}}\circ\pi_{1/1,d,2}]&=(d-1)\sigma_1(d)\left(\frac{1}{24}\Delta_{01,\{1,3\}}+\Delta_{11,\{1,3\}}\right)
\end{align*}
Applying Proposition \ref{m21_delta0,delta1} and \S\ref{int_numbers_m13}, we get no contributions to $\Delta_{01a}$ and $\Delta_{01b}$, a contribution of $2(d-1)\sigma_1(d)$ to $\Delta_{00}$, and a contribution of $-\frac{1}{24}(d-1)\sigma_1(d)$ to $\Xi_1$. Combining all of the above yields the needed intersection numbers.
\end{proof}

\begin{thm}\label{m21_adm_formula}
The class of $\pi_{2/1,d}$ in $A^2(\barM_{2,1})$ is:
\begin{align*}
&\left(-\frac{1}{12}d\sigma_1(d)+\frac{1}{12}\sigma_3(d)\right)\delta_{00}+\left(\frac{1}{12}\sigma_1(d)-\frac{1}{12}\sigma_3(d)\right)\delta_{01a}\\
+&\left(\left(-d-\frac{1}{12}\right)\sigma_1(d)+\frac{13}{12}\sigma_3(d)\right)\delta_{01b}+\left(2\sigma_3(d)-2d\sigma_1(d)\right)\xi_{1}+\left(4\sigma_3(d)-4\sigma_1(d)\right)\delta_{11}.
\end{align*}
In particular,
\begin{equation*}
\sum_{d\ge1}[\pi_{2/1,d}]q^d\in\Qmod\otimes A^2(\barM_{2,1}).
\end{equation*}
\end{thm}

\begin{proof}
Immediate from \S\ref{int_numbers_m21}; note that the dual graphs associated to all five boundary classes have automorphism group of order 2, except $\Delta_{00}$, which has automorphism group of order 8.
\end{proof}

The classes $\delta_{00},\delta_{01a},\delta_{01b}$ push forward to zero on $\barM_2$, and the classes $\xi_1,\delta_{11}$ push forward to $\delta_0,\delta_1$, respectively. Thus, Theorem \ref{m21_adm_formula} recovers Theorem \ref{main_thm_g2}. In addition, taking $d=2$ in Theorem \ref{m21_adm_formula} recovers \cite[Proposition 3.2.9]{vanzelm_thesis}.


\section{The $d$-elliptic locus on $\barM_3$}\label{m3_d-ell}

We now carry out the methods developed earlier and use the results above to compute the (unpointed) $d$-elliptic locus in genus 3, $[\pi_{3/1,d}]\in A^{2}(\barM_3)$. As $A_2(\cM_3)=0$ \cite[Theorem 1.9]{faber_thesis}, any test surface lies in one of the two boundary divisors:
\begin{enumerate}
\item[$(\Delta_0)$] $\barM_{2,2}\to\barM_{3}$
\item[$(\Delta_1)$] $\barM_{2,1}\times\barM_{1,1}\to\barM_{3}$
\end{enumerate} 

\subsection{Classification of Admissible Covers}

We will compute the intersection of the admissible locus with a general test surface $S$ in one of the two boundary divisors. The same arguments from before show that we need only consider the codimension 1 strata in $\Adm_{3/1,d}$, parametrizing covers whose targets have exactly one node. Moreover, the same dimension count shows that we may disregard the strata whose images in $\barM_3$ have dimension 2 or less, or equivalently whose general fiber under the map $\pi_{3/1,d}$ is positive-dimensional.

A similar analysis as in genus 2 yields seven topological types of covers in $\Adm_{3/1,d}$ that give nonzero contributions to general test surfaces, shown in Figures \ref{Fig:g3_Delta0-Delta12}, \ref{Fig:g3_Delta1-Delta12}, \ref{Fig:g3_Delta1-Delta13}, \ref{Fig:g3_Delta11-Delta14}, \ref{Fig:g3_Delta0-Delta0}, \ref{Fig:g3_Delta00-Delta0}, \ref{Fig:g3_Delta000-Delta0}.

\begin{figure}[!htb]
   \begin{minipage}{0.48\textwidth}
     \centering
     \includegraphics[width=.7\linewidth]{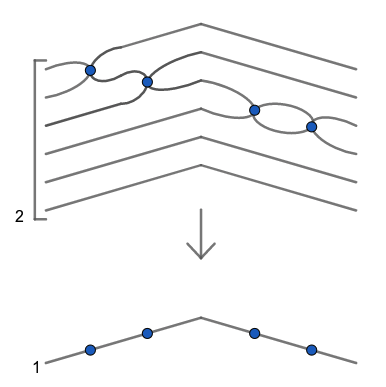}
     \caption{Cover of type $(\Delta_0,\Delta_{1,2})$}\label{Fig:g3_Delta0-Delta12}
   \end{minipage}\hfill
   \begin{minipage}{0.48\textwidth}
     \centering
     \includegraphics[width=.7\linewidth]{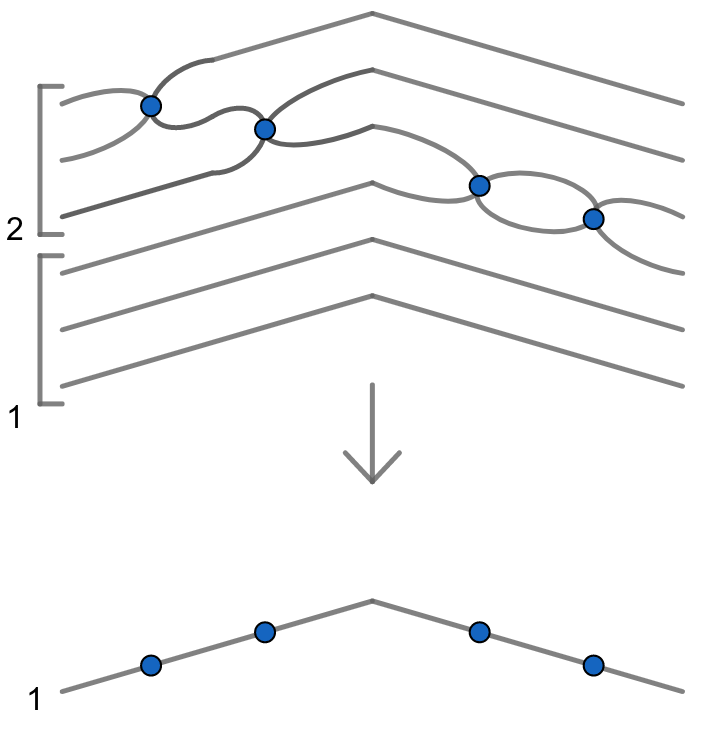}
     \caption{Cover of type $(\Delta_1,\Delta_{1,2})$}\label{Fig:g3_Delta1-Delta12}
   \end{minipage}
\end{figure}

\begin{figure}[!htb]
   \begin{minipage}{0.48\textwidth}
     \centering
     \includegraphics[width=.7\linewidth]{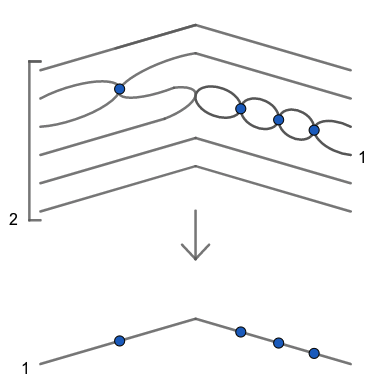}
     \caption{Cover of type $(\Delta_1,\Delta_{1,3})$}\label{Fig:g3_Delta1-Delta13}
   \end{minipage}\hfill
   \begin{minipage}{0.48\textwidth}
     \centering
     \includegraphics[width=.7\linewidth]{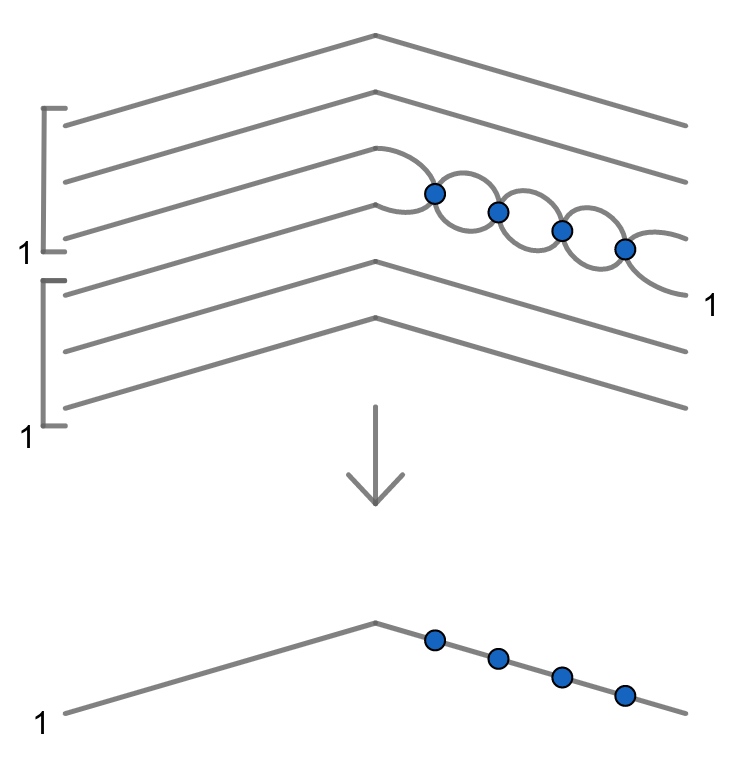}
     \caption{Cover of type $(\Delta_{11},\Delta_{1,4})$}\label{Fig:g3_Delta11-Delta14}
   \end{minipage}
\end{figure}

\begin{figure}[!htb]
   \begin{minipage}{0.48\textwidth}
     \centering
     \includegraphics[width=.55\linewidth]{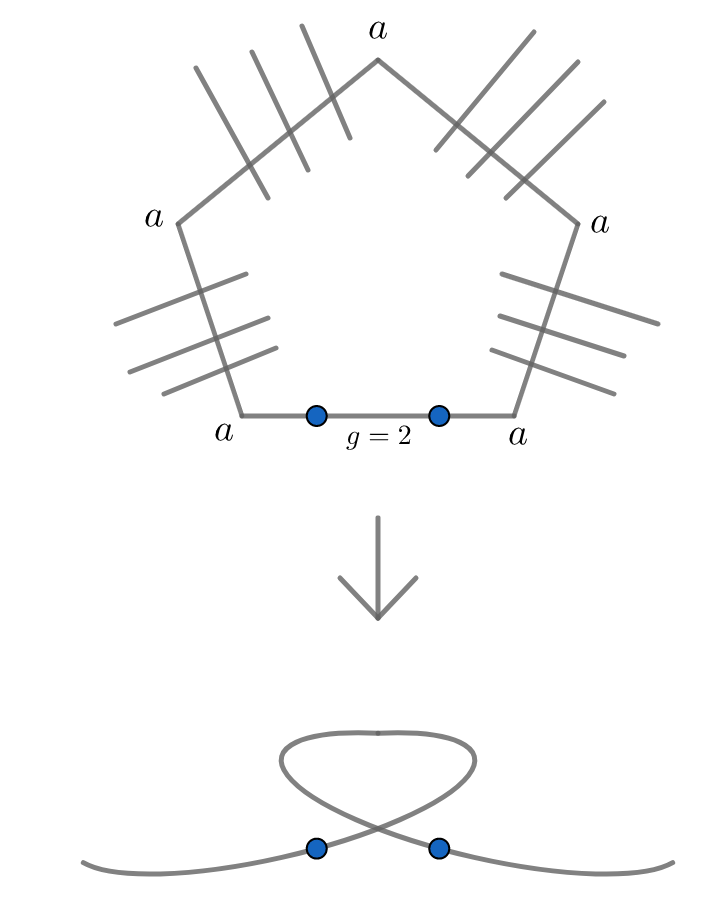}
     \caption{Cover of type $(\Delta_0,\Delta_0)$}\label{Fig:g3_Delta0-Delta0}
   \end{minipage}\hfill
   \begin{minipage}{0.48\textwidth}
     \centering
     \includegraphics[width=.85\linewidth]{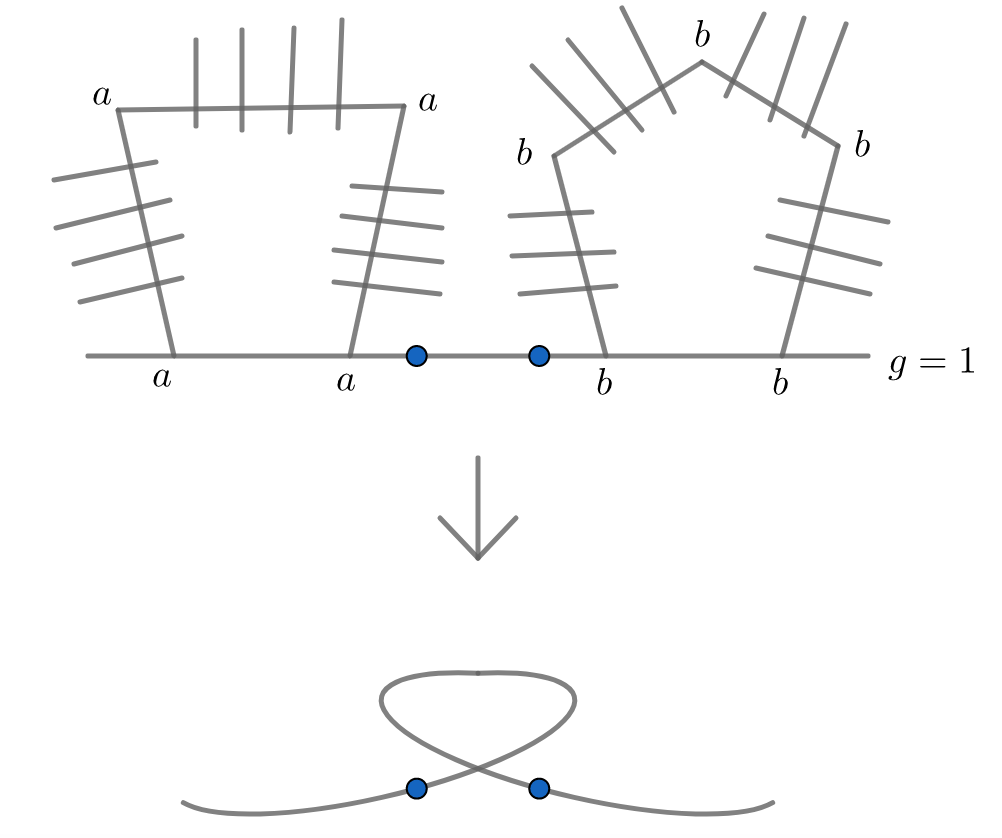}
     \caption{Cover of type $(\Delta_{00},\Delta_0)$}\label{Fig:g3_Delta00-Delta0}
   \end{minipage}
\end{figure}

\begin{figure}[!htb]
\includegraphics[width=.5\linewidth]{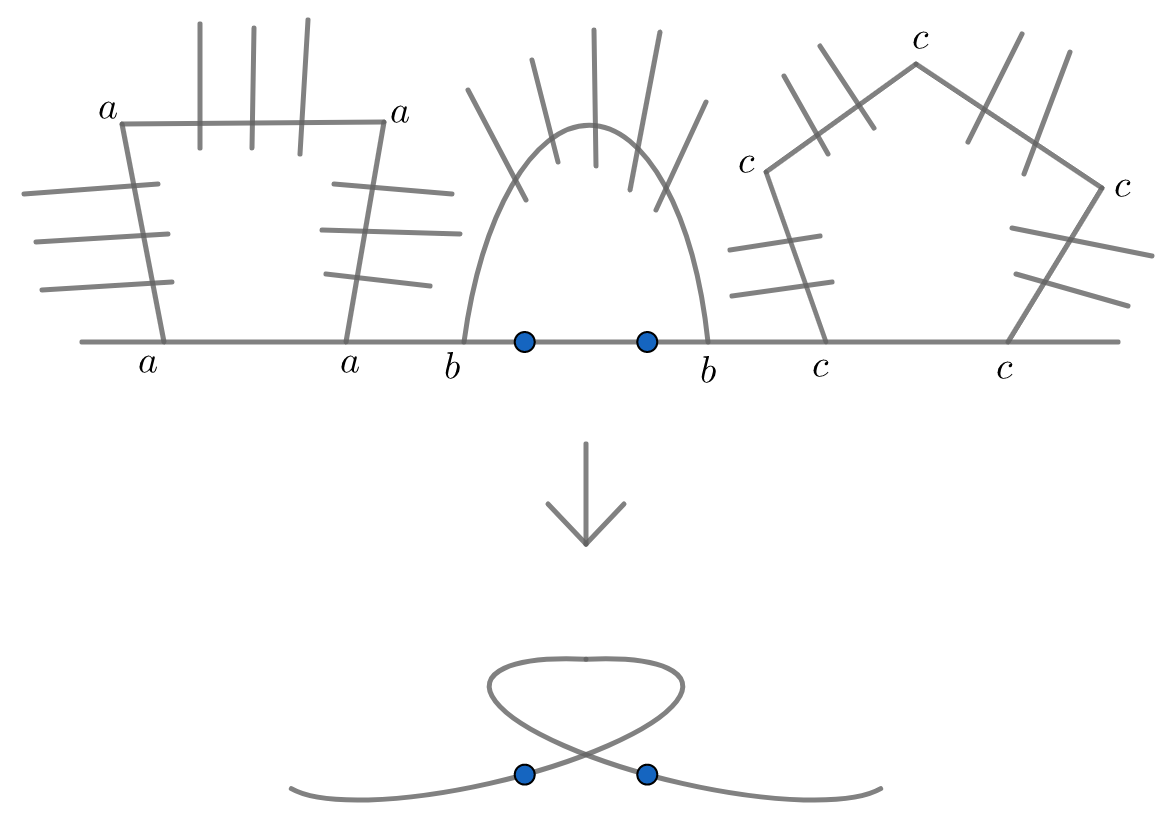}
\caption{Cover of type $(\Delta_{000},\Delta_0)$}\label{Fig:g3_Delta000-Delta0}
\end{figure}

\subsection{The case $[S]\in\Delta_1$}
Define the intersection $A_S$ by the Cartesian diagram
\begin{equation*}
\xymatrix{
A_S \ar[rr] \ar[d] & & \Adm_{3/1,d} \ar[d]^{\pi_{3/1,d}}\\
S \ar[r] & \barM_{2,1}\times\barM_{1,1} \ar[r] & \barM_{3}
}
\end{equation*}
We consider the contributions to $A_S$ from covers of the three possible types: $(\Delta_1,\Delta_{1,2})$, $(\Delta_1,\Delta_{1,3})$, and $(\Delta_{11},\Delta_{1,4})$.

\subsubsection{Type $(\Delta_1,\Delta_{1,2})$}

For any $d\ge1$, we define the space $\Adm_{2/1,d}^1$ as a functor by the Cartesian diagram
\begin{equation*}
\xymatrix{
\Adm_{2/1,d}^1 \ar[r] \ar[d]^{\pi^{1}_{2/1,d}} & \Adm_{2/1,d} \ar[d]^{\pi_{2/1,d}}\\
\barM_{2,1} \ar[r]^u & \barM_{2}
}
\end{equation*}
Here, $u:\barM_{2,1}\to\barM_{2}$ is the map forgetting the marked point. Generically, $\Adm_{2/1,d}^1$ parametrizes covers $f:X\to Y$ along with an arbitrary point of $X$. Define the map $\psi_{2/1,d}^{1}$ by the composition
\begin{equation*}
\Adm_{2/1,d}^{1}\to\Adm_{2/1,d}\to\barM_{1,2}\to\barM_{1,1}
\end{equation*}
where the middle map is $\psi_{2/1,d}$ and the last map forgets the second point.

As in \S\ref{adm_g2_intersect_delta1}, we have a diagram

\begin{equation*}
\xymatrix{
A^{(\Delta_1,\Delta_{1,2}),(d_1,d_2)}_{S} \ar[r] \ar[dd] & \Adm^1_{2/1,d_1}\times_{\Delta}\Adm_{1/1,d_2} \ar[rr] \ar[d] & &  \barM_{1,1}  \ar[d]^{\Delta} \\
 & \Adm^1_{2/1,d_1}\times\Adm_{1/1,d_2} \ar[rr]^(0.57){\psi_{2/1,d_1}^{1}\times\psi_{1/1,d_2}} \ar[d]^{\pi^1_{2/1,d_1}\times\pi_{1/1,d_2}} & & \barM_{1,1}\times\barM_{1,1} \\\
S \ar[r] & \barM_{2,1}\times\barM_{1,1} & 
}
\end{equation*}
where both squares are Cartesian. A geometric point of $A_S$ corresponding to a cover $f:X\to Y$ of type $(\Delta_1,\Delta_{1,2})$ gives rise to a geometric point of $A^{(\Delta_1,\Delta_{1,2}),(d_1,d_2)}_{S}$ (for some $d_1,d_2$) in an obvious way. Note, however, that the image of $[f]$ in $\barM_{1,1}\times\barM_{1,1}$ is of the form $([Y_1,q],[Y_1,y])$, where $Y_1\subset Y$ is the elliptic component, $q\in Y_1$ is one of the branch points of $f$, and $y\in Y_1$ is the node. While $q\neq y$, $[(Y_1,q)]$ and $[(Y_1,y)]$ are isomorphic via translation.

Owing to the contraction of the rational bridge of any admissible cover of type $(\Delta_1,\Delta_{1,2})$, each cover $f$ of type $(\Delta_1,\Delta_{1,2})$ appears with multiplicity 2 in $A_S$. In addition, each point $A^{(\Delta_1,\Delta_{1,2}),(d_1,d_2)}_{S}$ comes from $\binom{4}{2}=6$ points of $A_S$, due to the possible labelings of the branch points of the $d_1$-elliptic map.

Using the fact that $\Delta=[p\times\barM_{1,1}]+[\barM_{1,1}\times p]$, and applying the projection formula, we find:

\begin{prop}\label{m3_delta1,delta12}
The contribution to $A_S$ from covers of type $(\Delta_1,\Delta_{1,2})$ is:
\begin{equation*}
12\left(\int_{\barM_{2}\times\barM_{1,1}}u_{*}([S])\cdot\left(\sum_{d_1+d_2=d}\sigma_1(d_2)\left([[\pi_{2/1,d_1}(E)]\times\barM_{1,1}]+[[\pi_{2/1,d_1}]\times p]\right)\right)\right)
\end{equation*}
\end{prop}

\subsubsection{Type $(\Delta_1,\Delta_{1,3})$}

We have a Cartesian diagram
\begin{equation*}
\xymatrix{
A_S^{(\Delta_1,\Delta_{1,3})} \ar[rr] \ar[d] & & \Adm_{2/1,d} \ar[d]^{\pi_{2/1,d}} \\
S \ar[r] & \barM_{2,1}\times\barM_{1,1} \ar[r]^(0.6){\pr_1} & \barM_{2,1}
}
\end{equation*}
where $\pr_1: \barM_{2,1}\times\barM_{1,1}\to\barM_{2,1}$ is the projection. Given a point $([(C,q)],[(E,p)])$ of $\barM_{2,1}\times\barM_{1,1}$, the curve $C\cup_{p\sim q} E$ is the image under $\pi_{3/1,d}$ of a cover of type $(\Delta_1,\Delta_{1,3})$ if and only if there exists a $d$-elliptic map $g:C\to E$ ramified at $q$, in which we may glue $g$ to the unique double cover $E\to\bP^1$ ramified at the origin (and attach additional rational tails) to form an admissible cover whose source contracts to $C\cup_{p\sim q} E$. The transversality is straightforward, and we find that:

\begin{prop}\label{m3_delta1,delta13}
The contribution to $A_S$ from covers of type $(\Delta_1,\Delta_{1,3})$ is:
\begin{equation*}
24\left(\int_{\barM_{2,1}}\pr_{1*}([S])\cdot[\pi_{2/1,d}]\right)
\end{equation*}
\end{prop}
The factor of 24 comes from the $4!$ ways to label the ramification points.

\subsubsection{Type $(\Delta_{11},\Delta_{1,4})$}

For $d_1,d_2$ with $d_1+d_2=d$, we have a diagram
\begin{equation*}
\xymatrix{
A_S^{(\Delta_{11},\Delta_{1,4})} \ar[r] \ar[dd] & \barM_{1,2} \times \Adm_{1/1,d_1}\times_{\Delta}\Adm_{1/1,d_2} \ar[d] \ar[r] & \barM_{1,2}\times \Adm_{1/1,d_1}\times\Adm_{1/1,d_2} \ar[d]^{\id\times\pi_{1/1,d_1}\times\pi_{1/1,d_2}} \\
 & \barM_{1,2}\times\barM_{1,1} \ar[r]^(0.45){\id\times\Delta} & \barM_{1,2}\times \barM_{1,1}\times\barM_{1,1} \ar[d]^{\xi_1\times\id} \\
S \ar[rr] & & \barM_{2,1}\times\barM_{1,1}
}
\end{equation*}
where $\xi_1:\barM_{1,2}\times\barM_{1,1}\to\barM_{2,1}$ is the map defining the boundary divisor $\Delta_1$ in $\barM_{2,1}$, and both squares above are Cartesian.

\begin{prop}\label{m3_delta11,delta14}
The contribution to $A_S$ from covers of type $(\Delta_{11},\Delta_{1,4})$ is:
\begin{equation*}
24\left(\sum_{d_1+d_2=d}\sigma_1(d_1)\sigma_1(d_2)\right)\cdot\left(\int_{\barM_{2,1}\times\barM_{1,1}}[S]\cdot([\Delta_{01a}\times\barM_{1,1}]+[\Delta_1\times p])\right)
\end{equation*}
\end{prop}

\begin{proof}
Given covers $E_1\to E$ and $E_2\to E$ of degrees $d_1,d_2$, respectively, and a 2-pointed curve $(E',p_1,p_2)$ of genus 1, we construct an admissible cover of type $(\Delta_{11},\Delta_{1,4})$ by attaching $E'$ the $E_i$ at their origins along the $p_i$, mapping $E'\to\bP^1$ via the complete linear series $|\cO(p_1+p_2)|$, and labelling the ramification points in one of $4!=24$ ways. The transversality is straightforward. Decomposing the class of $\Delta$ as usual, we get the desired result; here the class $\Delta_{01a}\in A_2(\barM_{2,1})$ arises as the pushforward of $\barM_{1,2}\times p$ under $\xi_1$.
\end{proof}

\subsection{The case $[S]\in\Delta_0$}\label{m3_test_surface_nonsep}

The only such $S$ we will need is defined as follows. Let $C$ be a general curve of genus 2, and take $S=C\times C$. The map $S\mapsto\barM_{2,2}$ is defined by $(x,y)\mapsto [(C,x,y)]$, where if $x=y$, we interpret the image as the reducible curve with a 2-pointed rational curve attached at $x$. Clearly, covers of types $(\Delta_{00},\Delta_0)$ and $(\Delta_{000},\Delta_0)$ do not appear along $S$, so we do not give a general formula for contributions from such covers. Moreover, if $C$ is general, then it is not $d$-elliptic, so we also do not see covers of type $(\Delta_{0},\Delta_{1,2})$. Therefore, the only contributions to the intersection of $S$ with the admissible locus come from covers of type $(\Delta_0,\Delta_0)$.

\begin{prop}\label{CxC_int}
We have
\begin{equation*}
\int_{\barM_3}[C\times C]\cdot[\pi_{3/1,d}]=48(d\sigma_3(d)-\sigma_1(d)).
\end{equation*}
\end{prop}

\begin{proof}
This amounts to enumerating covers of type $(\Delta_0,\Delta_0)$ where the genus 2 component in the source is isomorphic to $C$; the result is then immediate from Proposition \ref{dd2222_g3} and a local computation identical to that of Proposition \ref{m2_delta0,delta0}. Indeed, we have
\begin{equation*}
\sum_{am=d}48(a^4-1)m=48d\left(\sum_{am=d}a^3\right)-48\left(\sum_{am=d}m\right)=48(d\sigma_3(d)-\sigma_1(d)).
\end{equation*}
\end{proof}

In the final computation, we will use the following:

\begin{prop}\label{class_of_CxC}
We have $[C\times C]=2(\Delta_{[1]}+\Delta_{[4]})$ in $A_2(\barM_3)$, where the classes on the right hand side are defined as in \S\ref{int_numbers_m3}.
\end{prop}

\begin{proof}
Because any two geometric points of $\barM_2$ are rationally equivalent, we may replace the general genus 2 curve $C$ by the reducible genus 2 curve $C_0$ obtained by gluing two nodal curves of arithmetic genus 1 together at a separating node. Then, the space $\overline{M}_{C_0,2}$ parametrizing two points on $C_0$ (that is, the fiber over $[C_0]$ of the forgetful map $\barM_{2,2}\to\barM_2$) has four components, corresponding to the choices of components of $C_0$ on which the marked points can lie. The two components of $\overline{M}_{C_0,2}$ for which the marked points lie on the same component of $C_0$ each contribute $\Delta_{[1]}$ to the class of $\overline{M}_{C_0,2}$, and the two components for which the marked points lie on opposite components each contribute $\Delta_{[4]}$.
\end{proof}

\subsection{The class of the admissible locus}

\begin{prop}\label{adm_m3_intersections}
We have:
\begin{align*}
\int_{\barM_3}[\pi_{3/1,d}]\cdot\Delta_{[1]}&=96(d-1)\sigma_1(d)\\
\int_{\barM_3}[\pi_{3/1,d}]\cdot\Delta_{[4]}&=24(d\sigma_{3}(d)-\sigma_1(d))-96(d-1)\sigma_1(d)\\
\int_{\barM_3}[\pi_{3/1,d}]\cdot\Delta_{[5]}&=24\left(\sum_{d_1+d_2=d}(2d_1-1)\sigma_1(d_1)\sigma_1(d_2)\right)\\
\int_{\barM_3}[\pi_{3/1,d}]\cdot\Delta_{[6]}&=0\\
\int_{\barM_3}[\pi_{3/1,d}]\cdot\Delta_{[8]}&=24\left(\sum_{d_1+d_2=d}d_1\sigma_1(d_1)\sigma_1(d_2)\right)-(d-1)\sigma_1(d)\\
\int_{\barM_3}[\pi_{3/1,d}]\cdot\Delta_{[10]}&=48\left(\sum_{d_1+d_2=d}\sigma_1(d_1)\sigma_1(d_2)\right)
\\
\int_{\barM_3}[\pi_{3/1,d}]\cdot\Delta_{[11]}&=24\left(\sum_{d_1+d_2+d_3=d}\sigma_1(d_1)\sigma_1(d_2)\sigma_1(d_3)\right)-\left(\sum_{d_1+d_2=d}\sigma_1(d_1)\sigma_1(d_2)\right)
\end{align*}
\end{prop}

\begin{proof}
We first deal with the surface classes $\Delta_{[i]}$ for $i=1,5,6,8,10,11$, which factor through $\barM_{2,1}\times\barM_{1,1}$. These are pushed forward from the following classes:
\begin{enumerate}
\item[($\Delta_{[1]}$)] $\Delta_{00}\times p$
\item[($\Delta_{[5]}$)] $\Gamma_{(5)}\times\barM_{1,1}$
\item[($\Delta_{[6]}$)] $\Gamma_{(6)}\times\barM_{1,1}$
\item[($\Delta_{[8]}$)] $\Xi_1\times p$
\item[($\Delta_{[10]}$)] $\Delta_{01a}\times p$ 
\item[($\Delta_{[11]}$)] (a) $\Delta_{[11]}\times p$ or (b) $\Gamma_{(11)}\times\barM_{1,1}$
\end{enumerate}

We summarize the contributions from covers of the three possible types in the table below, where we have applied Propositions \ref{m3_delta1,delta12}, \ref{m3_delta1,delta13}, and \ref{m3_delta11,delta14}, along with the intersection numbers of \S\ref{int_numbers_m21} and the intersection numbers with $d$-elliptic loci in genus 2 from Propositions \ref{adm_m2_intersections}, \ref{adm_m2E_intersections}, and \ref{adm_m21_intersections}. The last two rows correspond to the computation of the intersection of the admissible locus with $\Delta_{[11]}$, computed as the pushforwards of the two classes labelled (a) and (b) above.
\begin{center}
\begin{tabular}{| c | c | c | c |}
\hline
 & Type $(\Delta_1,\Delta_{1,2})$ & Type $(\Delta_1,\Delta_{1,3})$ & Type $(\Delta_{11},\Delta_{1,4})$ \\ \hline
$\Delta_{[1]}$ & 0 & $96(d-1)\sigma_1(d)$ & $0$ \\ \hline
$\Delta_{[5]}$ & \scalebox{0.75}{$48\left(\displaystyle\sum_{d_1+d_2=d}(d_1-1)\sigma_1(d_1)\sigma_1(d_2)\right)$} & 0 & \scalebox{0.75}{$24\left(\displaystyle\sum_{d_1+d_2=d}\sigma_1(d_1)\sigma_1(d_2)\right)$} \\ \hline
$\Delta_{[6]}$ & 0 & 0 & 0 \\ \hline
$\Delta_{[8]}$ & \scalebox{0.75}{$24\left(\displaystyle\sum_{d_1+d_2=d}(d_1-1)\sigma_1(d_1)\sigma_1(d_2)\right)$} & $-(d-1)\sigma_1(d)$ & \scalebox{0.75}{$24\left(\displaystyle\sum_{d_1+d_2=d}\sigma_1(d_1)\sigma_1(d_2)\right)$}\\ \hline
$\Delta_{[10]}$ & 0 & \scalebox{0.75}{$24\left(\displaystyle\sum_{d_1+d_2=d}\sigma_1(d_1)\sigma_1(d_2)\right)$} & \scalebox{0.75}{$24\left(\displaystyle\sum_{d_1+d_2=d}\sigma_1(d_1)\sigma_1(d_2)\right)$} \\ \hline
$\Delta_{[11]}$(a) & \scalebox{0.75}{$24\left(\displaystyle\sum_{d_1+d_2+d_3=d}\sigma_1(d_1)\sigma_1(d_2)\sigma_1(d_3)\right)$} & \scalebox{0.75}{$-\left(\displaystyle\sum_{d_1+d_2=d}\sigma_1(d_1)\sigma_1(d_2)\right)$} & 0\\ \hline
$\Delta_{[11]}$(b) & \scalebox{0.75}{$24\left(\displaystyle\sum_{d_1+d_2+d_3=d}\sigma_1(d_1)\sigma_1(d_2)\sigma_1(d_3)\right)$}  & 0 & \scalebox{0.75}{$-\left(\displaystyle\sum_{d_1+d_2=d}\sigma_1(d_1)\sigma_1(d_2)\right)$} \\ \hline
\end{tabular}
\end{center}
Combining the above yields six of the intersection numbers claimed; the seventh, of $[\pi_{3/1,d}]$ with $\Delta_{[4]}$, now follows from Propositions \ref{CxC_int} and \ref{class_of_CxC}.
\end{proof}

\begin{rem}
One can also implement the following check: the class $\Delta_{[7]}\in A_2(\barM_{3})$ is rationally equivalent to $\Delta_{[6]}$, so its intersection with the admissible locus should be zero. Using the fact that $\Delta_{[7]}$ is the pushforward of $\Delta_{01b}\times p$ from $\barM_{2,1}\times\barM_{1,1}$, we indeed find a contribution of 0 from type $(\Delta_1,\Delta_{1,2})$, a contribution of $\sum_{d_1+d_2=d}\sigma_1(d_1)\sigma_1(d_2)$ from type $(\Delta_1,\Delta_{1,3})$, and a contribution of $-\sum_{d_1+d_2=d}\sigma_1(d_1)\sigma_1(d_2)$ from type $(\Delta_{11},\Delta_{1,4})$.
\end{rem}

\begin{proof}[Proof of Theorem \ref{main_thm_g3}]
The result now follows from Proposition \ref{adm_m3_intersections}, along with the intersection numbers of \S\ref{int_numbers_m3} and the convolution formulas of \S\ref{qmod_prelim}.
\end{proof}


\appendix
\section{Quasi-modularity on $\barM_{2,2}$}\label{m22_appendix}

The quasimodularity for $d$-elliptic loci in genus 2 found in Theorems \ref{main_thm_g2} (forgetting marked branch points) and \ref{m21_adm_formula} (with one branch point) can in fact be upgraded to $\barM_{2,2}$, remembering both branch points of a $d$-elliptic cover. That is:

\begin{thm}\label{modularity_m22}
We have
\begin{equation*}
\sum_{d\ge1}[\pi_{2/1,d}]q^d\in\Qmod\otimes A^3(\barM_{2,2}).
\end{equation*}
\end{thm}

This result will propagate to quasimodular contributions to the $d$-elliptic locus on $\barM_4$, providing further evidence for Conjecture \ref{main_conj}. The method of proof is the same as above, using the fact that $A_2(\cM_{2,2})=0$ \cite[Lemma 1.14]{faber_thesis}; we do not carry out the full calculation. However, we point out one new aspect, that the contributions from admissible covers of certain topological types are not individually quasimodular, but the non-quasimodular contributions cancel in the sum.

Let $T\to\barM_{1,4}$ be a general boundary cycle of dimension 3. Consider the contributions to the intersection of $T$ with $\pi_{2/1,d}:\Adm_{2/1,d}\to\barM_{2,2}$ from covers of types $(\Delta_0,\Delta_0)$ and $(\Delta_{00},\Delta_0)$.

We have the following purely combinatorial lemma:

\begin{lem}\label{hurwitz_numbers}
Let $H(d,\lambda_1,\lambda_2,\lambda_3)$ denote the Hurwitz number counting degree $d$ covers (weighted by automorphisms) $f:C\to\bP^1$ branched over 3 points with ramification profiles $(\lambda_1,\lambda_2,\lambda_3)$, where we require $C$ connected. We have:
\begin{itemize}
\item[(a)] 
\begin{equation*}
H(d;(d),(d),(3,1^{d-3}))=\frac{(d-1)(d-2)}{6}
\end{equation*}
\item[(b)]
\begin{equation*}
H(d;(a,b),(a,b),(3,1^{d-3}))=
\begin{cases}
1\text{ if }a\neq b\\
0\text{ if }a=b
\end{cases}
\end{equation*}
\end{itemize}
\end{lem}

\begin{proof}
Let $\alpha\in S_d$ denote the cycle $(12\cdots d)$, and let $\beta=(1jk)$, for $j,k\in\{2,3,\ldots,d\}$ distinct. One readily checks that the product $\beta\alpha$ is a $d$-cycle if and only if $j<k$. There are $\binom{d-1}{2}$ choices of such $\beta$, but each Hurwitz factorization is then triple-counted owing to the simultaneous conjugations by powers of $\alpha$ sending 1 to $j,k$. The first formula follows.

For the second, let $\alpha$ be the permutation $(12\cdots a)(a+1\cdots d)$. We seek a 3-cycle $\beta$ for which $\beta\alpha$ also has cycle type $(a,b)$. In order for the corresponding branched cover to be a map of connected curves, $\beta$ cannot act trivially on either orbit of $\alpha$, so we may assume that $\beta$ acts nontrivially on $1,a+1$. If $a>b$, then we find $\beta=(1(a-b+1)(a+1))$, but if $a=b$, then no such $\beta$ exists.
\end{proof}

First, consider covers of type $(\Delta_0,\Delta_0)$. As in Propositions \ref{m2_delta0,delta0} and \ref{m21_delta0,delta0}, we have:

\begin{prop}\label{m22_delta0,delta0}
The contribution to the intersection of $T$ and $\pi_{2/1,d}$ from covers of type $(\Delta_0,\Delta_0)$ is 
\begin{equation*}
\sum_{am=d}\left(m\int_{\barM_{1,4}}[T]\cdot[\pi_{1/0,a}^{a,a,2,2}]\right).
\end{equation*}
\end{prop}

We now apply Proposition \ref{m22_delta0,delta0} with $T=\Delta_{3,4}$. The intersection $[T]\cdot[\pi_{1/0,a}^{a,a,2,2}]$ includes admissible covers formed by gluing a degree $d$ map $E\to\bP^1$ branched over three points with ramification indices $d,d,3$ to a degree 3 map $\bP^1\to\bP^1$ with ramification indices $3,2,2$, at the triple points in the source and target. Applying Lemma \ref{hurwitz_numbers}(a) and Proposition \ref{m22_delta0,delta0}, we find a contribution to $\int_{\barM_{2,2}}[T]\cdot[\pi_{2/1,d}]$ of
\begin{equation*}
\sum_{am=d}\frac{(a-1)(a-2)}{6}\cdot m=\left(\frac{1}{6}d+\frac{1}{3}\right)\sigma_1(d)-\frac{1}{2}d\tau(d),
\end{equation*}
where $\tau(d)$ denotes the number of divisors of $d$; the generating function for $d\tau(d)$ is not quasimodular.

On the other hand, consider contributions along $T$ from covers of type $(\Delta_{00},\Delta_0)$. We get a quasimodular contribution analogous to that of Proposition \ref{m21_delta00,delta0}, but we get a new contribution from admissible covers formed by gluing a degree $d$ map $\bP^1\to\bP^1$ branched over three points with ramification profiles $(a,b),(a,b),(3,1^{d-3})$ to a degree 3 map $\bP^1\to\bP^1$ with ramification indices $3,2,2$, at the triple points in the source and target. By Lemma \ref{hurwitz_numbers}(b) and the usual local computation, we get an additional contribution of 
\begin{align*}
&\sum_{am+bn=d}mb-\sum_{b(m+n)=d}mb\\
=&\sum_{d_1+d_2=d}\sigma_1(d_1)\sigma_1(d_2)-\sum_{bn'=d}\sum_{i=1}^{n'-1}ib\\
=&\sum_{d_1+d_2=d}\sigma_1(d_1)\sigma_1(d_2)-\sum_{bn'=d}b\left(\frac{n'(n'-1)}{2}\right)\\
=&\sum_{d_1+d_2=d}\sigma_1(d_1)\sigma_1(d_2)-\frac{1}{2}d\sigma_1(d)+\frac{1}{2}d\tau(d)
\end{align*}
In particular, the last term cancels out the non-quasimodular term from type $(\Delta_0,\Delta_0)$.

%
%
%
%

\end{document}